\documentclass[a4,10pt]{article}

\usepackage{wrapfig}

\usepackage{lineno}

\usepackage{amsmath}
\usepackage{amsthm}
\usepackage{amssymb}
\usepackage{enumerate}
\usepackage[pdftex]{graphicx}
\usepackage{float}

\usepackage{url}
\usepackage[colorlinks=true, linkcolor=blue, citecolor=blue, urlcolor=blue]{hyperref}
\usepackage{cite}

\usepackage[top=36truemm,bottom=35truemm,left=32truemm,right=32truemm]{geometry}

\usepackage{titlesec}
\titleformat*{\section}{\large\bfseries}

\theoremstyle{definition}
\newtheorem{theorem}{Theorem}[section]

\newtheorem{lemma}[theorem]{Lemma}

\newtheorem{proposition}[theorem]{Proposition}
\newtheorem{remark}[theorem]{Remark}
\numberwithin{equation}{section}

\allowdisplaybreaks

\begin{document}
  \title{\Large{Analytic Study of $p$-Bessel Functions: Fractional Calculus, \\Integral Representations, and Complex Extensions}}
  \author{Masaya Kitajima}
  \date{}
  \maketitle
  \begin{abstract}
  We present a systematic analytic study of the \textit{$p$-Bessel functions} $\mathcal{J}_{\omega,\varphi}^{[p]}$, a novel class of generalized Bessel functions arising from Fourier analysis on planar domains bounded by $p$-circles, including astroid-type shapes with $0<p\le2$ satisfying $2/p\in\mathbb{N}$. While previous work established Hardy-type oscillatory identities for these domains, expressing lattice point discrepancies via $p$-Bessel functions, the present paper focuses on the intrinsic analytic properties of the functions themselves. In particular, we (i) construct a hierarchical structure of $\{\mathcal{J}_{\omega,\varphi}^{[p]}\}_{\omega\ge0}$ using Erd\'{e}lyi–Kober-type fractional derivatives, (ii) derive explicit real-analytic integral representations and obtain asymptotic formulas on the coordinate axes, and (iii) extend the functions to the complex domain through Poisson-type integral formulas. These results establish $p$-Bessel functions as genuinely new oscillatory kernels, providing a rigorous framework for studying anisotropic oscillatory phenomena and laying the analytic foundation for applications in $p$-circle lattice point problems.\\
\textbf{Keywords:} Fourier analysis, oscillatory identity, Bessel function, Lam\'{e}'s curves, Poisson summation, fractional derivatives and integrals.\\
\textbf{2020 Mathematics Subject Classification:} 42A38, 33C10, 42B05, 26A33, 11P82.
\\
  \end{abstract}
  \section{Introduction and main results}
\hspace{13pt}In this paper, we study a class of generalized Bessel functions arising from Poisson summation (\cite{Stein-1971}, p251, Theorem 2.4) over a planar domain bounded by the \textit{$p$-circle}. This closed curve, defined for $p>0$ by $\{x \in \mathbb{R}^2 \mid |x_1|^p + |x_2|^p = r^p\}$, is also known as a Lamé curve or superellipse (a generalization of the circle). In particular, we focus on $p$-circles with $0<p\le 2$ satisfying $2/p \in \mathbb{N}$, which include both the circle and non-circular planar shapes such as the \textit{astroid}. Despite their simple algebraic definition, their associated identities and oscillatory structures have not been systematically studied. However, the condition $2/p\in\mathbb{N}$ ensures that these structures are suitable for the Fourier transform. Central to our investigation is a generalized Bessel function, which we call the \textit{$p$-Bessel functions of order $\omega\geq0$ and distorted angle $\varphi\in[0,2\pi)$}, defined by\vspace{-3pt}
\begin{equation}\label{p-Bessel-S}\vspace{-2pt}
\mathcal{J}_{\omega,\varphi}^{[p]}(r)
= \Bigl(\frac{2}{p}\Bigr)^{2+\omega} \frac{\pi}{\Gamma(\frac{1}{p})^2} 
\sum_{k=0}^{\infty} \frac{(-1)^k}{k! \, \Gamma(\frac{2}{p}(k+1) + \omega)}
\Bigl(\frac{r}{2}\Bigr)^{2k+\omega} \Phi_{k,\varphi}^{[p]}, 
\quad r \geq 0,
\end{equation}
\begin{equation}\label{Phi-co}
\text{with }\Phi_{k,\varphi}^{[p]} 
= \sum_{n=0}^{k} \frac{\bigl(\frac{2}{p}(k+1)-1\bigr)!}{n!\,(k-n)!} 
\left( \frac{B(\frac{2}{p}(n+\frac{1}{2}), \frac{2}{p}(k-n+\frac{1}{2}))}{B(n+\frac{1}{2}, k-n+\frac{1}{2})} \right) 
(\cos^{\frac{4}{p}}\varphi)^n (\sin^{\frac{4}{p}}\varphi)^{k-n}.
\end{equation}
Note that $B$ is the beta function and $\Gamma$ is the gamma function. \par
The structure of this definition consists of an outer series in the radial variable $r$, 
whose coefficients $\Phi_{k,\varphi}^{[p]}$ encode angular anisotropy through a finite combinatorial sum. Note that this condition on $p$ guarantees uniform convergence of the series representation on compact subsets of $\mathbb{R}_{\geq 0}$. These $p$-Bessel functions extend the classical Bessel functions, recovering them in the case $p=2$ (that is, $\mathcal{J}_{\omega,\varphi}^{[2]} = J_{\omega}$). They capture the anisotropic oscillatory behavior that emerges when rotational symmetry is broken. Figure 1 illustrates the dependence of the $p$-Bessel function on the distorted angle $\varphi$.

\begin{figure}[t]
    \centering
    \includegraphics[width=0.8\textwidth]{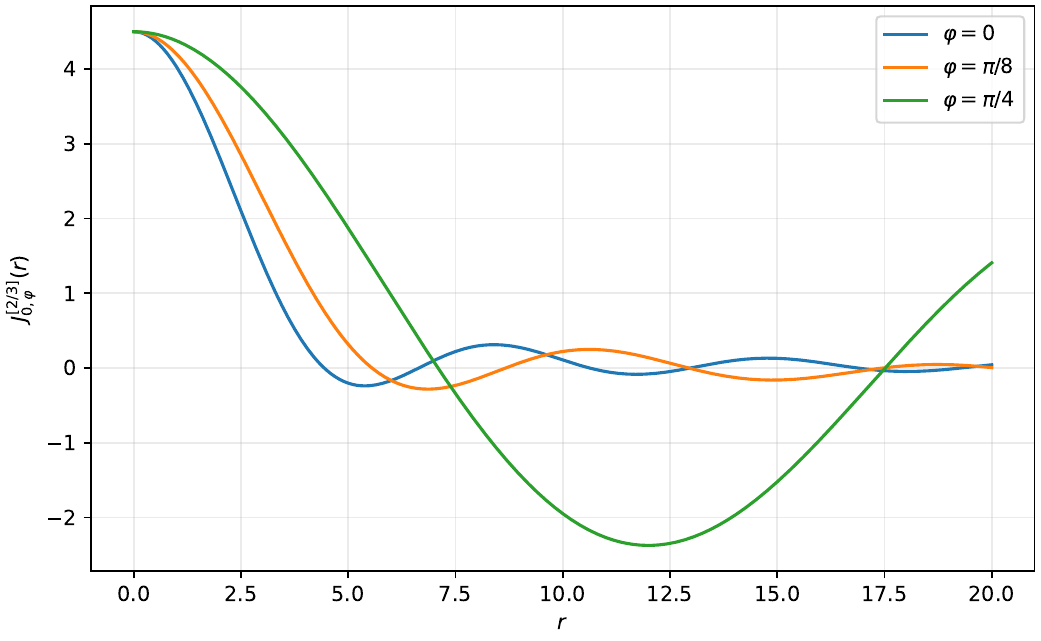}\vspace{-10pt}
    \caption{Plots of the $p$-Bessel function $\mathcal{J}^{[2/3]}_{0,\varphi}(r)$ for $\varphi=0$, $\pi/8$, and $\pi/4$.}
    \label{fig:pBessel_phi}
\end{figure}

\par\vspace{5pt}
  Our motivation for this study is strongly rooted in the lattice point problem for $p$-circle domains.\par
  Recalling the lattice point problem for the $p$-circle, we focus on the discrepancy between the area of the region bounded by the $p$-circle and the number of lattice points it contains, denoted by $N_p(r)$. Approximating the closed curve as a mosaic, as illustrated in \cite[Fig.1]{K3}, we define
\begin{equation}
P_p(r) := N_p(r) - \frac{2}{p} \frac{\Gamma(\frac{1}{p})^2}{\Gamma(\frac{2}{p})} r^2.
\end{equation}
The classical $p$-circle lattice point problem asks for an exponent $\alpha_p$ such that $
P_p(r) = \mathcal{O}(r^{\alpha_p})$ and $P_p(r) = \Omega(r^{\alpha_p})$, where, for functions $f$ and $g$, $f(t) = \mathcal{O}(g(t))$ (resp. $f(t) = \Omega(g(t))$) means $\limsup_{t\to\infty} |f(t)/g(t)| < +\infty$ (resp. $> 0$). \par
For $p = 2$, the case of a circle, this is known as the Gauss circle problem\cite{Gauss}. In 1917, G.H. Hardy\cite{Hardy-1917}, building on his work\cite{Hardy-1915} and \textit{Hardy's identity} (for example, see also \cite{Kratzel}, Theorem 3.12)
\begin{equation}\label{HI}
  P_{2}(r)=r\sum_{k=1}^{\infty}\frac{R(k)}{k^{\frac{1}{2}}}J_{1}(2\pi k^{\frac{1}{2}}r)\quad\text{with $R(k):=\#\{n\in\mathbb{Z}^{2}|\ |n|^{2}=k\}$},
 \end{equation}
conjectured that the infimum exponent in the $\mathcal{O}$-estimate for $P_2(r)$ is $1/2$ (Hardy's conjecture). While numerous improvements have been made since then the problem remains unsolved (the most recent result by X. Li and X. Yang \cite{XX} in 2023).\par
For $p > 2$, an effective approach exists based on a key theorem of E. Kr\"atzel (\cite{Kratzel}, Theorem 3.17 A), and the problem has been completely resolved for $p > 73/27$ by G. Kuba\cite{Kuba}.\par
By contrast, the range $0 < p < 2$ has received comparatively little attention, 
due in part to intrinsic analytic difficulties. In particular, applying Kr\"atzel's method to these cases leads to singularities in the main functions used, which significantly complicates the analysis. Nevertheless, different lattice point problems from ours in this range have been widely studied and remain an interesting area of research (see, for example, \cite{Laugesen1, Laugesen2}). \par\vspace{7pt}

Now, we turn our attention to Hardy's identity (\ref{HI}) and the Bessel functions, which play an important role in the lattice point problem of the circle domain. Hardy's identity shows that the oscillatory behavior of lattice point errors can be described using Bessel functions. This is a direct consequence of the preservation of spherical symmetry in the Fourier transform, which allows the oscillations to be captured by a single, direction-independent function. When the geometric object is generalized to the $p$-circle, however, this rotational symmetry is lost. As a result, the Fourier transform exhibits direction-dependent contributions, leading to a more intricate interference pattern in the lattice point errors.  \par
Building on the above, we employ the $p$-Bessel function (\ref{p-Bessel-S}) in previous work to derive a Hardy-type oscillatory identity for the \textit{astroid-type $p$-circles}, which generalizes (\ref{HI}) and reflects the underlying anisotropic geometry. This approach captures the oscillatory structures arising from the loss of rotational symmetry and provides an analytic framework for establishing Hardy-type identities across these regions.
\begin{theorem}[\textit{\cite[Theorem 1.2]{K3}; Hardy-type oscillatory identity for the astroid-type $p$-circle}]\label{GHI}\ \\\itshape{Let $p$ satisfy $2/p\in\mathbb{N}$ and a finite set $\mathcal{A}_{s}^{[p]}$ consist of distorted angles $\varphi$ corresponding to lattice points on $p$-circle of radius $s^{1/p}\ (\geq1)$. Specifically, $\mathcal{A}_{s}^{[p]}$ is denoted as follows $(\#\mathcal{A}_{s}^{[p]}\leq4[s^{\frac{1}{p}}])$. 
    \begin{equation}\label{mathcalA}
    \mathcal{A}_{s}^{[p]}:=\left\{\varphi\in[0,2\pi)\;\middle|\;\left(\mathrm{sgn}\left(\cos\varphi\right)s^{\frac{1}{p}}|\cos\varphi|^{\frac{2}{p}},\mathrm{sgn}\left(\sin\varphi\right)s^{\frac{1}{p}}|\sin\varphi|^{\frac{2}{p}}\right)\in\mathbb{Z}^{2}\right\}.
    \end{equation}
Here, $\# \mathcal{A}_{s}^{[p]}$ denotes the cardinality of $\mathcal{A}_{s}^{[p]}$. Then, the following holds for the counting measure $\mu$.}
      \begin{equation*}\label{HI-p}
    P_{p}(r)=\frac{p\Gamma(\frac{1}{p})^{2}}{2\pi}r\int_{1}^{\infty}\frac{1}{s^{\frac{1}{p}}}\Bigl(\sum\nolimits_{\varphi\in\mathcal{A}_{s}^{[p]}}\mathcal{J}_{1,\varphi}^{[p]}(2\pi s^{\frac{1}{p}}r)\Bigr)d\mu(s). 
    \end{equation*}
\end{theorem}\vspace{10pt}

Before presenting further analytic results, we emphasize that the $p$-Bessel functions are the central objects of this paper. While Theorem \ref{GHI} provides a Hardy-type oscillatory identity for astroid-type $p$-circles, it serves primarily as an application illustrating their relevance to lattice point problems. The main objective of this work is a systematic analytic study of the functions $\mathcal{J}_{\omega,\varphi}^{[p]}$ themselves. They encode the direction-dependent oscillations inherent to the $p$-circle and provide the analytic framework for understanding the oscillatory behavior in the Hardy-type identity. \par

This examination is crucial because the defining series of $\mathcal{J}_{\omega,\varphi}^{[p]}$, involving an outer infinite sum and an inner finite sum with coefficients $\Phi_{k,\varphi}^{[p]}$ that couple geometric and combinatorial data, reflects the intrinsic anisotropy of astroid-type $p$-circle domains. Understanding these structural properties lays the foundation for both Hardy-type identities and a broader analytic treatment of $p$-circle lattice point problems. \par\vspace{10pt}

A central feature of the present $p$-Bessel functions is the binomial-convolution structure of the coefficients $\Phi_{k,\varphi}^{[p]}$. As shown in our previous work (see the second half of Section 2.1 in \cite{K3}), by applying Gauss's duplication formula to the corresponding Gamma-function expression and using the Beta--Gamma relation, $\Phi_{k,\varphi}^{[p]}$ can be written as
\begin{align*}
\Phi_{k,\varphi}^{[p]}&=\sum_{n=0}^{k}\binom{k}{n}\frac{\Gamma(\frac{1}{p}+\frac{2}{p}n)}{\Gamma(n+\frac{1}{2})}\frac{\Gamma(\frac{1}{p}+\frac{2}{p}(k-n))}{\Gamma(k-n+\frac{1}{2})}(\cos^{\frac{4}{p}}\varphi)^{n}(\sin^{\frac{4}{p}}\varphi)^{k-n}\label{Phi-convolution}\\
&=\sum_{n=0}^{k}\binom{k}{n}a_{n} b_{k-n}\qquad\text{with }a_{n}:=\frac{\Gamma(\frac{1}{p}+\frac{2}{p}n)}{\Gamma(n+\frac{1}{2})}(\cos^{\frac{4}{p}}\varphi)^{n},\quad
b_{n}:=\frac{\Gamma(\frac{1}{p}+\frac{2}{p}n)}{\Gamma(n+\frac{1}{2})}(\sin^{\frac{4}{p}}\varphi)^{n}.
\end{align*}
Then, the exponential generating function of $\{\Phi_{k,\varphi}^{[p]}\}_{k\geq0}$ factors into two one-variable series:
\begin{align*}
\sum_{k=0}^{\infty}\frac{\Phi_{k,\varphi}^{[p]}}{k!}r^{k}&=\sum_{k=0}^{\infty}\biggl(\sum_{n=0}^{k}\frac{a_{n}b_{k-n}}{n!(k-n)!}\biggr)r^{k}=\sum_{n=0}^{\infty}\sum_{m=0}^{\infty}\frac{a_{n}b_{m}}{n!\ m!}r^{n+m}\notag\\
&=\biggl(\sum_{n=0}^{\infty}\frac{\Gamma(\frac{1}{p}+\frac{2}{p}n)}{\Gamma(n+\frac{1}{2})}\frac{(\cos^{\frac{4}{p}}\varphi\ r)^{n}}{n!}\biggr)\times\biggl(\sum_{m=0}^{\infty}\frac{\Gamma(\frac{1}{p}+\frac{2}{p}m)}{\Gamma(m+\frac{1}{2})}\frac{(\sin^{\frac{4}{p}}\varphi\ r)^{m}}{m!}\biggr)
\end{align*}
Thus, each factor on this rightmost side is of \textit{generalized Wright type}; more precisely, in the notation of \cite{Wright},
\begin{equation*}
\sum_{k=0}^{\infty}\frac{\Phi_{k,\varphi}^{[p]}}{k!}r^{k}={}_{1}\Psi_{1}\left[
\begin{matrix}
(\frac{1}{p},\frac{2}{p})\\[2pt]
(\frac{1}{2},1)
\end{matrix}
;\cos^{\frac{4}{p}}\varphi\ r\right]\times {}_{1}\Psi_{1}\left[
\begin{matrix}
(\frac{1}{p},\frac{2}{p})\\[2pt]
(\frac{1}{2},1)
\end{matrix}
;\sin^{\frac{4}{p}}\varphi\ r\right].
\end{equation*}\par
Moreover, the \textit{generalized multi-index Mittag--Leffler function} introduced by Saxena and Nishimoto\cite{Saxena-Nishimoto}, which is defined by
\begin{equation*}
E^{\gamma,\kappa}_{(\alpha_j,\beta_j)_{j=1}^{m}}(z):=\sum_{n=0}^{\infty}\frac{(\gamma)_{\kappa n}}{\prod_{j=1}^{m}\Gamma(\alpha_j n+\beta_j)}\frac{z^n}{n!},
\end{equation*}
shows that each of the two factors above is, more precisely, the $m=1$ case of such a function. Indeed, using $(\gamma)_{\kappa n}:=\Gamma(\gamma+\kappa n)/\Gamma(\gamma)$, we have
\begin{equation*}
\sum_{n=0}^{\infty}\frac{\Gamma(\frac{1}{p}+\frac{2}{p}n)}{\Gamma(n+\frac{1}{2})}\frac{z^{n}}{n!}=\Gamma\left(\frac{1}{p}\right)E^{\frac{1}{p},\frac{2}{p}}_{(1,\frac{1}{2})}(z).
\end{equation*}\par
Consequently, the exponential generating function of $\{\Phi_{k,\varphi}^{[p]}\}_{k\geq0}$ admits both \textit{generalized Wright and generalized multi-index Mittag--Leffler representations}, although $\Phi_{k,\varphi}^{[p]}$ itself is naturally described as a finite binomial convolution of the corresponding coefficients.\par
The above representation also clarifies the relation between the present functions and the classical Bessel functions. In particular, $\mathcal{J}_{\omega,\varphi}^{[2]}=J_{\omega}$, whereas for $p\neq2$ the dependence on the distorted angle $\varphi$ reflects the anisotropy of the underlying $p$-circle. We therefore use the terminology ``$p$-Bessel functions'' to emphasize their role as Bessel-type oscillatory kernels arising naturally in the Fourier analysis of $p$-radial domains. \par\vspace{10pt}

We next discuss their analytic properties and summarize the main contributions of this work.\par
First, we establish a hierarchical structure of the family $\{\mathcal{J}_{\omega,\varphi}^{[p]}\}_{\omega\ge0}$ through \textit{Erd\'{e}lyi–Kober-type fractional differential operators} $D_{0+;p,\eta}^{\gamma}$.
 \begin{theorem}[\textit{Erd\'{e}lyi–Kober-type fractional derivative}]\label{E-K-derJ}\ \\
  \itshape{Let $p$ satisfy $2/p\in\mathbb{N}$. Then, for $\omega\geq0$ and order of derivative $0<\gamma<1$, the following holds.}
  \begin{equation*}
    (D_{0+;p,(1-\frac{1}{p})\omega+\frac{2-\gamma}{p}-1}^{\gamma})\mathcal{J}_{\omega+\gamma,\varphi}^{[p]}(r)=\Bigl(\frac{r}{p}\Bigr)^{\gamma}\mathcal{J}_{\omega,\varphi}^{[p]}(r)\qquad\text{for }r>0.
  \end{equation*}
  \end{theorem}
This identity provides a real-analytic hierarchical framework, linking generalized Bessel functions of different orders and highlighting the natural role of fractional calculus in the $p$-circle setting.\par\vspace{5pt}
Next, we derive explicit real-analytic integral representations valid for $r \ge 0$, which are suitable for investigating axis-dependent asymptotic behavior:
\begin{theorem}[\textit{Integral representation}]\label{matcal-J-int0}
\itshape{Let $p$ satisfy $2/p\in\mathbb{N}$, then the following holds.}
\begin{align*}
     \mathcal{J}_{\omega,\varphi}^{[p]}(r)=
        \frac{(\frac{2}{p})^{2}r^{\omega}}{p^{\omega-1}\Gamma(\omega)\Gamma(\frac{1}{p})^{2}}\int_{0}^{1}\biggl(\int_{0}^{1}\cos(r&(s(1-u^{p})\cos^{\frac{2}{p}}\varphi)^{\frac{1}{p}})s^{\frac{1}{p}-1}(1-s)^{\omega-1}ds\biggr)\\
        \times&\cos(ru\sin^{\frac{2}{p}}\varphi)(1-u^{p})^{\frac{1}{p}+\omega-1}du\quad\text{for }\omega>0,
\end{align*}
\begin{equation*}
        \mathcal{J}_{0,\varphi}^{[p]}(r)=\frac{4}{p\Gamma(\frac{1}{p})^{2}}\int_{0}^{1}\cos(r(1-u^{p})^{\frac{1}{p}}\cos^{\frac{2}{p}}\varphi)\cos(ru\sin^{\frac{2}{p}}\varphi)(1-u^{p})^{\frac{1}{p}-1}du.
  \end{equation*}
\end{theorem}
These integral formulas, defined as double- and single-variable integrals, are particularly suitable for real-analytic asymptotic studies. \par\vspace{5pt}
\begin{theorem}[\textit{Asymptotic formula on the coordinate axes}]\label{AFCA}
\itshape{Let $p$ satisfy $2/p\in\mathbb{N}\setminus\{2\}$ and $\omega\geq0$. For
$\varphi_{\mathrm{axis}}\in\{0,\frac{\pi}{2},\pi,\frac{3\pi}{2}\}$, let $x_{\mathrm{axis}}$ denote the corresponding point on the coordinate axes with $|x_{\mathrm{axis}}|=r$. Then, as $r\to\infty$,}
\begin{equation*}
  J_{\omega}^{[p]}(x_{\mathrm{axis}})=\frac{4\Gamma(p+1)}{p^{\omega+1}\Gamma(\omega+\frac{1}{p}-1)\Gamma(\frac{1}{p})}r^{\omega-p-1}\sin\frac{\pi p}{2}+\frac{4p^{\frac{1}{p}-2}}{\Gamma(\frac{1}{p})}r^{-\frac{1}{p}}\cos\Bigl(r-\frac{\pi}{2}(\frac{1}{p}+\omega)\Bigr)+\mathcal{O}(r^{-\beta(p,\omega)}),
\end{equation*}
where $\beta(p,\omega):=\min\{1+2p-\omega,1+1/p\}$. In particular, this result indicates that an asymptotic estimate uniform with respect to the distorted angle $\varphi$ should depend on the order $\omega$.
\end{theorem}

Finally, we extend the series definition to the complex domain and derive an alternative integral representation in terms of the entire function
\begin{equation*}
\cos_{\varphi}^{[p]}(z) := \sum_{k=0}^{\infty} \frac{(-1)^{k}}{(2k)!!} \left( \frac{\sqrt{\pi} \, \Phi_{k,\varphi}^{[p]}}{\Gamma(\frac{1}{p}(2k+1)) 2^{k}} \right) z^{2k}, \quad z\in\mathbb{C}.
\end{equation*}
\begin{theorem}[\textit{Poisson-type integral representation for $p$-Bessel Functions}]\label{GPIR}\ \\
\itshape{Let $p$ satisfy $2/p$ is a positive odd integer, then the following holds.}
\begin{equation*}
\mathcal{J}_{\omega,\varphi}^{[p]}(z) = \frac{\sqrt{\pi}(\frac{2}{p})^{2+\omega}2}{\Gamma(\frac{1}{p})^{2}\Gamma(\omega+\frac{1}{p})} \Bigl(\frac{z}{2}\Bigr)^{\omega} \int_{0}^{\frac{\pi}{2}} \cos_{\varphi}^{[p]}(z \cos^{\frac{2}{p}}\theta) \, \sin^{2\omega}\theta \, (\cos\theta \sin\theta)^{\frac{2}{p}-1} \, d\theta.
\end{equation*}
\end{theorem}
This complex-analytic form confirms that $\mathcal{J}_{\omega,\varphi}^{[p]}$ is holomorphic on $z\in \mathbb{C}\setminus\mathbb{R}_{-}$, and all results derived for $r \ge 0$, including fractional derivative relations and real-analytic integral formulas, remain valid in the complex setting. In particular, this integral representation, being a one-variable function, is suitable for future studies using residue calculus and contour deformation, whereas the integral representation (Theorem \ref{matcal-J-int0}), defined as a two-variable integral, is primarily useful for asymptotic analysis on the real axis. \par\vspace{5pt}

Taken together, these observations establish $\mathcal{J}_{\omega,\varphi}^{[p]}$ as a genuinely new class of oscillatory kernels, forming an anisotropic analogue of classical Bessel functions in Fourier analysis. Their study not only generalizes Hardy-type lattice point identities but also provides a comprehensive analytic framework for $p$-circle lattice point problems, integrating fractional calculus, real and complex integral representations, and asymptotic analysis.\par\vspace{7pt}

This paper is structured as follows. In the next section, we prove identities involving Erd\'{e}lyi–Kober-type fractional derivatives (Theorem \ref{E-K-derJ}). Section 3 focuses on asymptotic behavior at infinity. It derives an integral representation (Theorem \ref{matcal-J-int0}) and, in particular, obtains an asymptotic formula on the coordinate axes (Theorem \ref{AFCA}). Finally, we extend the $p$-Bessel functions to the complex domain and derive a Poisson-type integral representation (Theorem \ref{GPIR}).

   \section{Erd\'{e}lyi–Kober-Type Fractional Derivative Identities (Proof of Theorem \ref{E-K-derJ})} 
  \hspace{13pt}In this section, we show that the $p$-Bessel functions are naturally related through the Erd\'{e}lyi–Kober-type fractional derivative operators $D_{0+;p,\eta}^{\gamma}$. First, we recall two lemmas used in the proof: an integral formula that raises the order and a differential formula that lowers the order.
  \begin{lemma}[\textit{\cite[(2.14)]{K3}; Order-raising integral formula} (\textit{analogue of that for $J_{\omega}$})]\label{IF}\ \\
  \itshape{Let $p$ satisfy $2/p\in\mathbb{N}$. Then, for $\omega\geq0$, $\gamma>0$, and $\varphi\in[0,2\pi)$, the following holds.}
  \begin{equation*}
    \mathcal{J}_{\omega+\gamma,\varphi}^{[p]}(r)=\frac{r^{\gamma}}{p^{\gamma-1}\Gamma(\gamma)}\int_{0}^{1}\mathcal{J}_{\omega,\varphi}^{[p]}(\tau r)\tau^{(p-1)\omega+1}(1-\tau^{p})^{\gamma-1}d\tau\qquad\text{for }r>0.
  \end{equation*}
  \end{lemma}
  \begin{lemma}[\textit{\cite[Proposition 2.6]{K3}; Order-lowering differential formula} (\textit{analogue of that for $J_{\omega}$})]\label{DF}
  \itshape{Let $p$ satisfy $2/p\in\mathbb{N}$. Then, for $\omega\geq0$ and $\varphi\in[0,2\pi)$, the following holds.}
    \begin{equation*}
    \frac{d}{dr}r^{1+(p-1)\omega}\mathcal{J}_{\omega+1,\varphi}^{[p]}(r)=r^{1+(p-1)\omega}\mathcal{J}_{\omega,\varphi}^{[p]}(r)\qquad\text{for }r>0.
    \end{equation*}
  \end{lemma}
  Next, for $\alpha\in\mathbb{C}\setminus\{0\}$ satisfying $\mathrm{Re}(\alpha)>0$, with $n:=[\mathrm{Re}(\alpha)]+1$, $p>0$, and $\eta\in\mathbb{C}$, we introduce the Erd\'{e}lyi–Kober–type fractional integral and fractional derivative (\cite[(2.6.1), (2.6.29)]{FDE}).
  \begin{align}
    (I_{0+;p,\eta}^{\gamma})f(r)&:=\frac{p}{\Gamma(\gamma)}\int_{0}^{1}\frac{\tau^{p(\eta+1)-1}f(\tau r)}{(1-\tau^{p})^{1-\gamma}}d\tau,\label{E-K-I}\\
    (D_{0+;p,\eta}^{\gamma})f(r)&:=r^{-p\eta}\Bigl(\frac{1}{pr^{p-1}}\frac{d}{dr}\Bigr)r^{p(1+\eta)}(I_{0+;p,\eta+\gamma}^{1-\gamma})f(r)\qquad\text{for }r>0.\label{E-K-D}
  \end{align}
  \begin{proof}[Proof of Theorem \ref{E-K-derJ}]
  Let $p$ satisfy $2/p \in \mathbb{N}$, and assume that $\omega \geq 0$ and $0 < \gamma < 1$. In this case, for the $p$-Bessel functions $\mathcal{J}_{\omega+\gamma,\varphi}^{[p]}$, applying the fractional derivative of order $\gamma$ with $\eta=(1-1/p)\omega+(2-\gamma)/p-1$ in (\ref{E-K-D}), it follows from the definition of the fractional integral (\ref{E-K-I}) that it can be expressed, as in 
  \begin{align*}
    (D&_{0+;p,(1-\frac{1}{p})\omega+\frac{2-\gamma}{p}-1}^{\gamma})\mathcal{J}_{\omega+\gamma,\varphi}^{[p]}(r)\\
    &=r^{(1-p)\omega+p+(\gamma-2)}\Bigl(\frac{r^{1-p}}{p}\frac{d}{dr}\Bigr)\frac{p\ r^{(p-1)\omega+(2-\gamma)}}{\Gamma(1-\gamma)}\int_{0}^{1}\frac{\tau^{(p-1)\omega+1+(p-1)\gamma}\mathcal{J}_{\omega+\gamma,\varphi}^{[p]}(\tau r)}{(1-\tau^{p})^{\gamma}}d\tau\\
    &=r^{(1-p)\omega+\gamma-1}\frac{d}{dr}\frac{r^{(p-1)\omega+2-\gamma}}{\Gamma(1-\gamma)}\int_{0}^{1}\mathcal{J}_{\omega+\gamma,\varphi}^{[p]}(\tau r)\ 
    \tau^{(p-1)(\omega+\gamma)+1}{(1-\tau^{p})^{(1-\gamma)-1}}d\tau\\
    &=\frac{r^{(1-p)\omega+\gamma-1}}{p^{\gamma}}\frac{d}{dr}r^{(p-1)\omega+1}\frac{r^{1-\gamma}}{p^{(1-\gamma)-1} \Gamma(1-\gamma)}\int_{0}^{1}\mathcal{J}_{\omega+\gamma,\varphi}^{[p]}(\tau r)\ 
    \tau^{(p-1)(\omega+\gamma)+1}{(1-\tau^{p})^{(1-\gamma)-1}}d\tau,
  \end{align*}
  in terms of the usual first-order differentiation and integration. Furthermore, by Lemma \ref{IF} and Lemma \ref{DF}, which describe the properties of $\mathcal{J}_{\omega,\varphi}^{[p]}$, we obtain the following concise representation. 
  \begin{align*}
    (D_{0+;p,(1-\frac{1}{p})\omega+\frac{2-\gamma}{p}-1}^{\gamma})\mathcal{J}_{\omega+\gamma,\varphi}^{[p]}(r)
    &=\Bigl(\frac{r}{p}\Bigr)^{\gamma}r^{-1-(p-1)\omega}\frac{d}{dr}r^{1+(p-1)\omega}\mathcal{J}_{(\omega+\gamma)+(1-\gamma),\varphi}^{[p]}(r)\\
    &=\Bigl(\frac{r}{p}\Bigr)^{\gamma}r^{-1-(p-1)\omega}r^{1+(p-1)\omega}\mathcal{J}_{\omega,\varphi}^{[p]}(r).
  \end{align*}
  As a result, we obtain the following order-lowering differential formula for $\mathcal{J}_{\omega,\varphi}^{[p]}$ in the sense of the Erd\'{e}lyi–Kober-type fractional derivative. 
  \begin{equation}\label{E-K-derJ+}
    (D_{0+;p,(1-\frac{1}{p})\omega+\frac{2-\gamma}{p}-1}^{\gamma})\mathcal{J}_{\omega+\gamma,\varphi}^{[p]}(r)=\Bigl(\frac{r}{p}\Bigr)^{\gamma}\mathcal{J}_{\omega,\varphi}^{[p]}(r)\qquad\text{for }0<\gamma<1.
  \end{equation}
  \end{proof}
  
  \begin{remark}This formula can be regarded as the counterpart of the following order-raising formula for $\mathcal{J}_{\omega,\varphi}^{[p]}$ in the sense of the Erd\'{e}lyi–Kober-type fractional integral (\cite[(2.9)]{K3}).
  \begin{equation}\label{E-K-intJ}
    (I_{0+;p,(1-\frac{1}{p})\omega+\frac{2}{p}-1}^{\gamma})\mathcal{J}_{\omega,\varphi}^{[p]}(r)=\Bigl(\frac{p}{r}\Bigr)^{\gamma}\mathcal{J}_{\omega+\gamma,\varphi}^{[p]}(r)\qquad\text{for }\gamma>0.
  \end{equation}
  Thus, the formulas obtained in this section show that the Erd\'{e}lyi–Kober fractional operators naturally describe the order-shifting structure of the $p$-Bessel functions.
  \end{remark}
  \vspace{10pt}
  Moreover, as an application of (\ref{E-K-intJ}), we derive an additional formula for $\mathcal{J}_{\omega,\varphi}^{[p]}$. Specifically, we define $\eta(p,\omega):=(1-\frac{1}{p})\omega+\frac{2}{p}-2$, with $E$ denoting the identity operator. By applying the first-order Erd\'{e}lyi–Kober-type fractional derivative, we obtain the following fractional differential equation.
  \begin{proposition}\itshape{Let $p>0$ satisfy $2/p\in\mathbb{N}$, then the following Erd\'{e}lyi–Kober–type fractional differential equation admits $u(r):=\mathcal{J}_{\omega,\varphi}^{[p]}(r)$ as a solution.}
  \begin{equation*}
    pr^{p}\Bigl(D_{0+;p,\eta(p,\omega)}^{1}\Bigr)u(r)+r\frac{d}{dr}\Bigl(I_{0+;p,\eta(p,\omega)+1}^{1}-E\Bigr)u(r)+(p-1)(\omega-2)\Bigl(I_{0+;p,\eta(p,\omega)+1}^{1}-E\Bigr)u(r)=0.
  \end{equation*}
  \end{proposition}
  \begin{proof}
  For the $p$-Bessel functions $\mathcal{J}_{\omega,\varphi}^{[p]}$, applying the fractional derivative of order 1 with $\eta(p,\omega):=(1-\frac{1}{p})\omega+\frac{2}{p}-2$ in (\ref{E-K-D}), it follows from the definition of the fractional integral (\ref{E-K-I}) that it can be expressed, as in 
  \begin{align*}
    \Bigl(&D_{0+;p,\eta(p,\omega)}^{1}\Bigr)\mathcal{J}_{\omega,\varphi}^{[p]}(r)\\
    &=r^{-p\eta(p,\omega)}\Bigl(\frac{1}{pr^{p-1}}\frac{d}{dr}\Bigr)\frac{1}{pr^{p-1}}\frac{d}{dr}r^{p(1+\eta(p,\omega))}\Bigl(I_{0+;p,\eta(p,\omega)+1}^{1}\Bigr)\mathcal{J}_{\omega,\varphi}^{[p]}(r)\\
    &=r^{-p\eta(p,\omega)}\Bigl(\frac{1}{pr^{p-1}}\frac{d}{dr}\Bigr)\frac{1}{pr^{p-1}}\frac{d}{dr}r^{p(1+\eta(p,\omega))}\frac{p}{\Gamma(1)}\int_{0}^{1}\frac{\tau^{p(\eta(p,\omega)+2)-1}\mathcal{J}_{\omega,\varphi}^{[p]}(\tau r)}{(1-\tau^{p})^{0}}d\tau\\
    &=\frac{1}{pr^{(p-1)(\omega-1)}}\Bigl(\frac{d}{dr}\Bigr)\frac{1}{pr^{p-1}}\frac{d}{dr}r^{-p}r^{1+(p-1)\omega}p\biggl(r\int_{0}^{1}\mathcal{J}_{\omega,\varphi}^{[p]}(\tau r)\tau^{(p-1)\omega+1}(1-\tau^{p})^{1-1}d\tau \biggr).
  \end{align*}
  Then, from Lemma \ref{IF}, Lemma \ref{DF}, and (\ref{E-K-intJ}), the desired expression is obtained as follows.
  \begin{align*}
    \Bigl(&D_{0+;p,\eta(p,\omega)}^{1}\Bigr)\mathcal{J}_{\omega,\varphi}^{[p]}(r)\\
    &=\frac{1}{pr^{(p-1)(\omega-1)}}\Bigl(\frac{d}{dr}\Bigr)\frac{1}{r^{p-1}}\frac{d}{dr}r^{-p}r^{1+(p-1)\omega}\mathcal{J}_{\omega+1,\varphi}^{[p]}(r)\\
    &=\frac{1}{pr^{(p-1)(\omega-1)}}\Bigl(\frac{d}{dr}\Bigr)\frac{1}{r^{p-1}}\Bigl(-pr^{-(p+1)}\cdot r^{1+(p-1)\omega}\mathcal{J}_{\omega+1,\varphi}^{[p]}(r)+r^{-p}\cdot r^{1+(p-1)\omega}\mathcal{J}_{\omega,\varphi}^{[p]}(r)\Bigr)\\
    &=\frac{1}{pr^{(p-1)(\omega-1)}}\frac{d}{dr}\biggl(-r^{(p-1)(\omega-2)}\Bigl(\frac{p}{r}\Bigr)\mathcal{J}_{\omega+1,\varphi}^{[p]}(r)+r^{(p-1)(\omega-2)}\mathcal{J}_{\omega,\varphi}^{[p]}(r)\biggr)\\
    &=\frac{-1}{pr^{(p-1)(\omega-1)}}\frac{d}{dr}r^{(p-1)(\omega-2)}\Bigl(I_{0+;p,\eta(p,\omega)+1}^{1}-E\Bigr)\mathcal{J}_{\omega,\varphi}^{[p]}(r)\\
    &=\frac{-r^{(p-1)(\omega-2)-1}}{pr^{(p-1)(\omega-1)}}\biggl((p-1)(\omega-2)\Bigl(I_{0+;p,\eta(p,\omega)+1}^{1}-E\Bigr)\mathcal{J}_{\omega,\varphi}^{[p]}(r)+r\frac{d}{dr}\Bigl(I_{0+;p,\eta(p,\omega)+1}^{1}-E\Bigr)\mathcal{J}_{\omega,\varphi}^{[p]}(r)\biggr).
  \end{align*}
  \end{proof}  
\section{Integral Representations and Asymptotics at Infinity}
  \hspace{13pt}In this section, we emphasize that the asymptotic formula for Bessel functions at infinity plays a crucial role in the lattice point problem. Under restricted conditions, we then study the asymptotic behavior (Theorem \ref{AFCA}) of our $p$-Bessel functions via oscillatory integral representations (Theorem \ref{matcal-J-int0}), as two of the main results of this paper.
  \subsection{\normalsize{Importance of p-Bessel Function Asymptotics in Lattice Point Identities}}
  \hspace{13pt}The asymptotic formula for Bessel functions (\cite{Watson}, p199; (1)) 
  \begin{equation}\label{J-asym.}
    J_{\omega}(r)=\sqrt{\frac{2}{\pi r}}\cos\Bigl(r-\frac{2\omega+1}{4}\pi\Bigr)+\mathcal{O}(r^{-\frac{3}{2}})\quad\text{for }r\to\infty
  \end{equation}
   plays an important role in deriving conjectures from certain identities related to the classical circle lattice point problem and in obtaining new results when combined with these identities. \par\vspace{5pt}
  For example, given that Hardy's identity (\ref{HI}) conditionally converges, by formally applying the asymptotic formula (\ref{J-asym.}), G.H. Hardy and E. Landau in \cite{Hardy-1917} conjectured that
    \begin{equation*}
    P_{2}(r)=\mathcal{O}(r^{\frac{1}{2}+\varepsilon}),\quad \text{but not }\mathcal{O}(r^{\frac{1}{2}})\quad\text{as }r\to\infty
  \end{equation*}
  for any sufficiently small $\varepsilon>0$ (\textit{Hardy's conjecture}). \par\vspace{5pt}  
  Another example is the identity due to S. Kuratsubo and E. Nakai \cite[(2.6)]{Kuratsubo-2022}, which gave a harmonic analytic claim equivalent to the Hardy's conjecture (see \cite[Theorem 7.1]{Kuratsubo-2022}).
  \begin{equation}\label{D-J}
      D_{\beta}(s:x)-\mathcal{D}_{\beta}(s:x)=s^{\beta+1}2^{\beta+1}\pi\sum_{ n\in\mathbb{Z}^{2}\setminus\{0\}}\frac{J_{\beta+1}(2\pi\sqrt{s}|x-n|)}{(2\pi\sqrt{s}|x-n|)^{\beta+1}}\qquad\text{if }\beta>\frac{1}{2}.
  \end{equation}
  Note that for $\beta>-1,\ s>0$, and $x\in\mathbb{R}^{2}$, these functions are defined by
  \begin{equation*}\label{D}
    D_{\beta}(s:x):=\frac{1}{\Gamma(\beta+1)}\sum_{|m|^{2}<s}(s-|m|^{2})^{\beta}
    e^{2\pi ix\cdot m},\quad 
    \mathcal{D}_{\beta}(s:x):=\frac{1}{\Gamma(\beta+1)}\int_{|\xi|^{2}<s}
    (s-|\xi|^{2})^{\beta}e^{2\pi ix\cdot \xi}d\xi,
  \end{equation*}
  in particular, $D_{0}(r^{2}:0)-\mathcal{D}_{0}(r^{2}:0)=P_{2}(r)$ holds. \par
  In this context, the asymptotic formula (\ref{J-asym.}) helps identify the range of variables $\beta$ for which this equation (\ref{D-J}) holds. \par\vspace{5pt}
  Therefore, from these two examples, it is evident that the asymptotic behavior of $p$-Bessel functions plays a crucial role in studying the lattice point problem for general $p$-circles.
  
  \subsection{\normalsize{Uniform Asymptotics of the 0-Order p-Bessel Function and Axis Dominance}}
  \hspace{13pt}For $p$ satisfying $2/p\in\mathbb{N}\setminus\{1,2\}$, the asymptotic estimate in $\mathcal{O}$ notation for the $p$-Bessel function of order zero is already known from our previous results (see \cite[Theorem 1.5]{K2}). 
  \begin{proposition}[\textit{Uniform asymptotics with respect to $\varphi$}]\label{uni-zero-ord}\ \\ 
    \itshape{Let $p$ satisfy $2/p\in\mathbb{N}\setminus\{1,2\}$. Then, the following holds uniformly with respect to $\varphi\in[0,2\pi)$.}\vspace{-5pt}
    \begin{equation*}\vspace{3pt}
      \mathcal{J}_{0,\varphi}^{[p]}(r)=\mathcal{O}(r^{-\frac{p}{2}})\qquad\text{as }r\to\infty.
    \end{equation*}
  \end{proposition}
  \begin{remark}
    Among $p$-circles with $0<p<1$, we have already seen that when $2/p$ is a natural number, the series representation of the $p$-Bessel functions converges uniformly on compact sets (see \cite[Proposition 2.1]{K3}, ). The same condition also guarantees the non-degeneracy of the phase function in the oscillatory integral representation of the $p$-Bessel functions. This highlights that the condition is both distinctive and essentially unavoidable in the analysis of $p$-circles with cusps.
  \end{remark}
  \par\vspace{10pt}
  By this proposition, we see that for $p$ such that $2/p$ is a natural number greater than or equal to $3$, the $p$-Bessel function of order zero decays uniformly at infinity in $\mathbb{R}^{2}$ with order $-p/2$.
On the other hand, as shown by the theorem in our previous paper cited above (see the following proposition), if the distorted angle $\varphi$ is restricted to an arbitrary compact subset in each quadrant of $\mathbb{R}^{2}$, then a sharper decay estimate of order $r^{-1/2}$ holds. \par
These results suggest that the slower decay in the uniform estimate is associated with the behavior near the coordinate axes. We therefore focus on the coordinate axes and analyze the oscillatory integral representation in order to clarify this behavior for general orders $\omega$.
  \begin{proposition}[\textit{\cite[Theorem 1.4]{K2}}]\label{cpt.uni}
  \itshape{Let $0<p<1$ or $p=2$. Then, the following holds uniformly with respect to $\varphi$ on every compact subset of $[0,2\pi)\setminus\{0,\frac{\pi}{2},\pi,\frac{3}{2}\pi\}$.}
  \begin{equation*}
    \mathcal{J}_{0,\varphi}^{[p]}(r)=\mathcal{O}(r^{-\frac{1}{2}})\qquad\text{as }r\to\infty.
  \end{equation*}
  \end{proposition}
  
  \subsection{\normalsize{Integral Representations of the p-Bessel Functions (Proof of Theorem \ref{matcal-J-int0})}}
  \hspace{13pt}In this subsection, for $p$ satisfying $2/p\in\mathbb{N}$, we derive the integral representation of the $p$-Bessel functions (Theorem \ref{matcal-J-int0}) from the two-variable extension of $J_{\omega}^{[p]}(x)$, which generalizes the previously one-variable function $\mathcal{J}_{\omega,\varphi}^{[p]}(r)$ by allowing the formerly fixed angle $\varphi$ to vary (\cite[(1.4)]{K3}). Here,
  \begin{equation*}
  x = (\mathrm{sgn}(\cos\varphi) r|\cos\varphi|^{2/p}, \mathrm{sgn}(\sin\varphi) r|\sin\varphi|^{2/p})\quad\text{and }|x|_{p}:=(|x_{1}|^{p}+|x_{2}|^{p})^{\frac{1}{p}},
  \end{equation*}
  \begin{equation}\label{2-J}
    J_{\omega}^{[p]}(x)=\frac{(\frac{2}{p})^{2}|x|_{p}^{\omega}}
      {p^{\omega}\Gamma(\frac{1}{p})^{2}}\sum_{m_{1}=0}^{\infty}\sum_{m_{2}=0}^{\infty}\frac{(-1)^{m_{1}+m_{2}}}{\Gamma(\frac{2}{p}(m_{1}+m_{2}+1)+\omega)}\frac{\Gamma(\frac{2m_{1}+1}{p})\Gamma(\frac{2m_{2}+1}{p})}{(2m_{1})!\ (2m_{2})!}x_{1}^{2m_{1}}x_{2}^{2m_{2}}.
  \end{equation}
  This representation will serve as a useful tool in the study of the asymptotic behavior of the $p$-Bessel functions of general order $\omega \geq 0$. In particular, in the next subsection, the asymptotic formula for these functions along the coordinate axes is obtained by using the oscillatory integral representation derived below.\par
  \begin{proof}[Proof of Theorem \ref{matcal-J-int0}]
   By absolute convergence of the series and Fubini's theorem, we may interchange the order of summation and integration as follows.
  \begin{align*}
      J_{\omega}^{[p]}(x)&=\frac{(\frac{2}{p})^{2}|x|_{p}^{\omega}}
      {p^{\omega}\Gamma(\frac{1}{p})^{2}}\sum_{m_{1}=0}^{\infty}\sum_{m_{2}=0}^{\infty}\frac{(-1)^{m_{1}+m_{2}}x_{1}^{2m_{1}}x_{2}^{2m_{2}}\Gamma(\frac{2}{p}m_{1}+\frac{1}{p})}{(2m_{1})!\ (2m_{2})!}\biggl(\frac{\Gamma(\frac{2}{p}m_{2}+\frac{1}{p})}{\Gamma(\frac{2}{p}(m_{1}+m_{2}+1)+\omega)}\biggr)\\
      &=\frac{(\frac{2}{p})^{2}|x|_{p}^{\omega}}{p^{\omega}\Gamma(\frac{1}{p})^{2}}\sum_{m_{1}=0}^{\infty}\sum_{m_{2}=0}^{\infty}\frac{(-1)^{m_{1}+m_{2}}x_{1}^{2m_{1}}x_{2}^{2m_{2}}\Gamma(\frac{2}{p}m_{1}+\frac{1}{p})}{(2m_{1})!\ (2m_{2})!\ \Gamma(\frac{2}{p}m_{1}+\frac{1}{p}+\omega)}\int_{0}^{1}t^{\frac{2}{p}m_{1}+\frac{1}{p}+\omega-1}(1-t)^{\frac{2}{p}m_{2}+\frac{1}{p}-1}dt\\
      &=\frac{(\frac{2}{p})^{2}|x|_{p}^{\omega}}{p^{\omega}\Gamma(\frac{1}{p})^{2}}\int_{0}^{1}\biggl(\sum_{m_{1}=0}^{\infty}\frac{(-1)^{m_{1}}\Gamma(\frac{2}{p}m_{1}+\frac{1}{p})(x_{1}t^{\frac{1}{p}})^{2m_{1}}}{(2m_{1})!\ \Gamma(\frac{2}{p}m_{1}+\frac{1}{p}+\omega)}\biggr)\biggl(\sum_{m_{2}=0}^{\infty}\frac{(-1)^{m_{2}}}{(2m_{2})!}(x_{2}(1-t)^{\frac{1}{p}})^{2m_{2}}\biggr)\\
      &\hspace{280pt}\times t^{\frac{1}{p}+\omega-1}(1-t)^{\frac{1}{p}-1}dt.
  \end{align*}\par
  Furthermore, for $\omega>0$, by using $\Gamma(2m_{1}/p+1/p)/\Gamma(2m_{1}/p+1/p+\omega)=\int_{0}^{1}s^{2m_{1}/p+1/p-1}(1-s)^{\omega-1}ds/\Gamma(\omega)$,\vspace{-5pt}
  \begin{align*}
 \sum_{m_{1}=0}^{\infty}\frac{(-1)^{m_{1}}\Gamma(\frac{2}{p}m_{1}+\frac{1}{p})(x_{1}t^{\frac{1}{p}})^{2m_{1}}}{(2m_{1})!\ \Gamma(\frac{2}{p}m_{1}+\frac{1}{p}+\omega)}&=\frac{1}{\Gamma(\omega)}\int_{0}^{1}\biggl(\sum_{m_{1}=0}^{\infty}\frac{(-1)^{m_{1}}}{(2m_{1})!}(x_{1}(st)^{\frac{1}{p}})^{2m_{1}}\biggr)s^{\frac{1}{p}-1}(1-s)^{\omega-1}ds\\
 &=\frac{1}{\Gamma(\omega)}\int_{0}^{1}\cos(x_{1}(st)^{\frac{1}{p}})s^{\frac{1}{p}-1}(1-s)^{\omega-1}ds
 \end{align*}
  can be expressed, and (although $x_{1}$ and $x_{2}$ are symmetric, for convenience) it can be combined into the following integral representation. 
  \begin{align}\label{J-int0>0}
     J_{\omega}^{[p]}(x)=
        \frac{(\frac{2}{p})^{2}|x|_{p}^{\omega}}{p^{\omega}\Gamma(\omega)\Gamma(\frac{1}{p})^{2}}\int_{0}^{1}\biggl(\int_{0}^{1}\cos&(x_{1}(st)^{\frac{1}{p}})s^{\frac{1}{p}-1}(1-s)^{\omega-1}ds\biggr)\notag\\
        \times&\cos(x_{2}(1-t)^{\frac{1}{p}})t^{\frac{1}{p}+\omega-1}(1-t)^{\frac{1}{p}-1}dt\qquad\text{for }\omega>0,
\end{align}
\begin{equation}\label{J-int0=0}
        J_{0}^{[p]}(x)=\frac{(\frac{2}{p})^{2}}{\Gamma(\frac{1}{p})^{2}}\int_{0}^{1}\cos(x_{1}t^{\frac{1}{p}})\cos(x_{2}(1-t)^{\frac{1}{p}})t^{\frac{1}{p}-1}(1-t)^{\frac{1}{p}-1}dt.
  \end{equation}\par
  Moreover, by making the change of variables $u=(1-t)^{1/p}$ (that is, $(1-t)^{1/p-1}dt=-pdu$, $t=1-u^{p}$), it can also be expressed in the following form. 
  \begin{align}\label{J-int>0}
     J_{\omega}^{[p]}(x)=
        \frac{(\frac{2}{p})^{2}|x|_{p}^{\omega}}{p^{\omega-1}\Gamma(\omega)\Gamma(\frac{1}{p})^{2}}\int_{0}^{1}\biggl(\int_{0}^{1}\cos&(x_{1}(s(1-u^{p}))^{\frac{1}{p}})s^{\frac{1}{p}-1}(1-s)^{\omega-1}ds\biggr)\notag\\
        \times&\cos(x_{2}u)(1-u^{p})^{\frac{1}{p}+\omega-1}du\qquad\text{for }\omega>0,
\end{align}
\begin{equation}\label{J-int=0}
        J_{0}^{[p]}(x)=\frac{4}{p\Gamma(\frac{1}{p})^{2}}\int_{0}^{1}\cos(x_{1}(1-u^{p})^{\frac{1}{p}})\cos(x_{2}u)(1-u^{p})^{\frac{1}{p}-1}du.
  \end{equation}
  \end{proof}
  
  \subsection{\normalsize{Asymptotic Formula on the Coordinate Axes (Proof of Theorem \ref{AFCA})}}
  \hspace{13pt}In this subsection, we derive an asymptotic formula for the $p$-Bessel function of a non-negative order along the coordinate axes (Theorem \ref{AFCA}) by analyzing the contributions from the endpoints of its integral representation. In particular, we apply the oscillatory analogue of \textit{Watson's lemma} to the endpoints and combine the resulting contributions to clarify the asymptotic behavior on the axes.
\begin{proof}[Proof of Theorem \ref{AFCA}]
Let $p$ be a positive real number satisfying $2/p \in\mathbb{N}\setminus\{2\}$, and let $\omega\geq0$ denote the order of the function. We consider the asymptotic behavior of $J_{\omega}^{[p]}$ along the axes $\varphi_{\mathrm{axis}}\in\{0,\pi/2,\pi,3\pi/2\}$ as $r\to\infty$. Among the corresponding points $x_{\mathrm{axis}}\in\{(r,0),(0,r),(-r,0),(0,-r)\}$, taking into account the choice of axis, the form of the integral representation, and the symmetry, we restrict our attention to the case $x_{\mathrm{axis}}=(0,r)$ in what follows. Hence, the integral representations (\ref{J-int0>0}) and (\ref{J-int0=0}) can be rewritten as follows.
  \begin{equation*}
    J_{\omega}^{[p]}(x_{\mathrm{axis}})=\frac{(\frac{2}{p})^{2}r^{\omega}}{p^{\omega-1}\Gamma(\omega+\frac{1}{p})\Gamma(\frac{1}{p})}\int_{0}^{1}\cos(ru)(1-u^{p})^{\frac{1}{p}+\omega-1}du \qquad\text{for }\omega\geq0.
  \end{equation*}
  Then, since the phase function of this oscillatory integral has no stationary points in the interior of the integration interval, the asymptotic evaluation of the oscillatory integral reduces to the analysis of the endpoint contributions rather than to an application of the stationary phase method. Consequently, the leading contribution appears to be governed by the behavior near the endpoints. \par\vspace{5pt}
  Now, fix $0<\delta<1/2$, and choose a smooth cutoff function $\chi_{0}\in C^{\infty}([0,1])$ such that 
\begin{equation*}
  \chi_{0}(u)=1\quad\text{for }0\leq u\leq\delta,\qquad \chi_{0}(u)=0\quad\text{for }2\delta\leq u\leq1.
\end{equation*}
Furthermore, we define $\chi_{1}(u):=1-\chi_{0}(u)$, then $\chi_{0}+\chi_{1}=1$ holds. Hence, we can set
  \begin{align}
    J_{\omega}^{[p]}(x_{\mathrm{axis}})&=C_{\omega}^{[p]}r^{\omega}(I_{0}(r)+I_{1}(r))\quad\text{with } C_{\omega}^{[p]}:=\frac{(\frac{2}{p})^{2}}{p^{\omega-1}\Gamma(\omega+\frac{1}{p})\Gamma(\frac{1}{p})},\label{axis-int}\\
    I_{j}(r)&:=\int_{0}^{1}\cos(ru)\chi_{j}(u)(1-u^{p})^{\alpha}du\quad \text{for }j=0\text{ or }1.\notag
  \end{align}
  Note that $\chi_{j}(u)(1-u^{p})^{\alpha}$ with $\alpha:=1/p+\omega-1$ is the amplitude function. \par\vspace{5pt}
  We first consider the contribution from the endpoint $u=1$. Let $\widetilde{\chi}_{1}(t):=\chi_{1}(1-t)$ for $t:=1-u$. Since $\chi_{1}(u)=0$ for $0\leq u\leq\delta$ and $\chi_{1}(u)=1$ for $2\delta\leq u\leq1$, we have 
  \begin{equation*}
  \widetilde{\chi}_{1}(t)=1\quad\text{for }0\leq t\leq1-2\delta,\qquad \widetilde{\chi}_{1}(t)=0 \quad\text{for }1-\delta\leq t\leq1.
\end{equation*}
  \par
  Moreover, from the binomial expansion $(1-t)^{p}=1-pt+\mathcal{O}(t^{2})$ as $t\to0+$, we obtain
\begin{equation*}
  (1-u^{p})^{\alpha}=(pt+\mathcal{O}(t^{2}))^{\alpha}=(pt)^{\alpha}(1+\mathcal{O}(t))^{\alpha}=p^{\alpha}t^{\alpha}(1+\mathcal{O}(t))=p^{\alpha}t^{\alpha}+\mathcal{O}(t^{\alpha+1})\quad\text{as }t\to0+.
\end{equation*}
  In addition, by the oscillatory analogue of \textit{Watson's lemma} (see \cite[(2.3.20)]{NIST-DLMF}), for $\beta>-1$ and $\chi\in C^{\infty}([0,\gamma])$ satisfying $\chi(\tau)=1$ near $\tau=0$ and $\chi(\tau)=0$ near $\tau=\gamma$, we obtain 
\begin{equation}\label{WL}
  \int_{0}^{\gamma}\cos(r\tau)\chi(\tau)\tau^{\beta}d\tau=\frac{\Gamma(\beta+1)}{r^{\beta+1}}\cos\frac{\pi}{2}(\beta+1)+\mathcal{O}(r^{-\beta-2})\quad\text{as }r\to\infty.
\end{equation}
  Similarly, $\int_{0}^{\gamma}\sin(r\tau)\chi(\tau)\tau^{\beta}d\tau=(\Gamma(\beta+1)/r^{\beta+1})\sin\frac{\pi}{2}(\beta+1)+\mathcal{O}(r^{-\beta-2})$ also holds. Combining these results, we can express
\begin{equation*}
  I_{1}(r)=p^{\alpha}\int_{0}^{1}\cos(r(1-t))\widetilde{\chi}_{1}(t) t^{\alpha}dt+\int_{0}^{1}\cos(r(1-t))\widetilde{\chi}_{1}(t)\mathcal{O}(t^{\alpha+1})dt,
\end{equation*}\vspace{-10pt}
\begin{align*}
  \int_{0}^{1}\cos(r(1-t))\widetilde{\chi}_{1}(t)t^{\alpha}dt&=\cos r\int_{0}^{1}\cos(rt)\widetilde{\chi}_{1}(t)t^{\alpha}dt+\sin r\int_{0}^{1}\sin(rt)\widetilde{\chi}_{1}(t)t^{\alpha}dt\\
  &=\frac{\Gamma(\alpha+1)}{r^{\alpha+1}}\Bigl(\cos r\cos\frac{\pi}{2}(\alpha+1)+\sin r\sin\frac{\pi}{2}(\alpha+1)\Bigr)+\mathcal{O}(r^{-\alpha-2})\\
  &=\frac{\Gamma(\alpha+1)}{r^{\alpha+1}}\cos\Bigl(r-\frac{\pi}{2}(\alpha+1)\Bigr)+\mathcal{O}(r^{-\alpha-2}).
\end{align*}  
Since $\alpha+1=1/p+\omega$, we finally obtain the following result.
\begin{align}
  C_{\omega}^{[p]}r^{\omega}I_{1}(r)&=\frac{(\frac{2}{p})^{2}r^{\omega}p^{\frac{1}{p}+\omega-1}\Gamma(\frac{1}{p}+\omega)}{p^{\omega-1}\Gamma(\omega+\frac{1}{p})\Gamma(\frac{1}{p})r^{\frac{1}{p}+\omega}}\cos\Bigl(r-\frac{\pi}{2}(\frac{1}{p}+\omega)\Bigr)+r^{\omega}\cdot\mathcal{O}(r^{-\frac{1}{p}-\omega-1})\notag\\
  &=\frac{4p^{\frac{1}{p}-2}}{\Gamma(\frac{1}{p})}r^{-\frac{1}{p}}\cos\Bigl(r-\frac{\pi}{2}(\frac{1}{p}+\omega)\Bigr)+\mathcal{O}(r^{-\frac{1}{p}-1})\quad\text{as }r\to\infty.\label{u1_endpoint}
\end{align}\par
On the other hand, we next consider the contribution from the endpoint $u=0$. \par

Applying (\ref{WL}) to the terms of the binomial expansion
\begin{equation}\label{BE}
  (1-u^{p})^{\alpha}=1-\alpha u^{p}+\frac{\alpha(\alpha-1)}{2}u^{2p}+\mathcal{O}(u^{3p}),
\end{equation}
the constant term ($\beta=0$) vanishes, and the following holds as for the second and third term ($\beta=p$ and $2p$). 
\begin{align}
  \int_{0}^{1}\cos(ru)\chi_{0}(u)u^{\beta}du&=\frac{\Gamma(\beta+1)}{r^{\beta+1}}\cos\frac{\pi}{2}(\beta+1)+\mathcal{O}(r^{-\beta-2}) \quad\text{for }\beta=p\text{ or }2p,\notag\\
  C_{\omega}^{[p]}r^{\omega}I_{0}(r)&=-\alpha C_{\omega}^{[p]}\frac{\Gamma(p+1)}{r^{p+1-\omega}}\cos\frac{\pi}{2}(p+1)+\mathcal{O}(r^{\omega-2p-1})\notag\\
  &=\frac{4(\frac{1}{p}+\omega-1)\Gamma(p+1)}{p^{\omega+1}\Gamma(\omega+\frac{1}{p})\Gamma(\frac{1}{p})}r^{\omega-p-1}\sin\frac{\pi p}{2}+\mathcal{O}(r^{\omega-2p-1})\quad\text{as }r\to\infty.\label{u0_endpoint}
\end{align}
For $p=2$, particularly, by (\ref{BE}), $\cos(2k+1)\pi/2=0$ for $k\in\mathbb{N}_{0}$ and (\ref{WL}), we have
\begin{equation}\label{anyN}
  C_{\omega}^{[2]}r^{\omega}I_{0}(r)=\mathcal{O}(r^{-N})\quad\text{ for any natural number }N, \text{ as }r\to\infty.
\end{equation}
\par
Thus, by (\ref{axis-int}), (\ref{u1_endpoint}) and (\ref{u0_endpoint}), we obtain the asymptotic formula for the axis.
\begin{equation*}
  J_{\omega}^{[p]}(x_{\mathrm{axis}})=\frac{4\Gamma(p+1)}{p^{\omega+1}\Gamma(\omega+\frac{1}{p}-1)\Gamma(\frac{1}{p})}r^{\omega-p-1}\sin\frac{\pi p}{2}+\frac{4p^{\frac{1}{p}-2}}{\Gamma(\frac{1}{p})}r^{-\frac{1}{p}}\cos\Bigl(r-\frac{\pi}{2}(\frac{1}{p}+\omega)\Bigr)+\mathcal{O}(r^{-\beta(p,\omega)}).
\end{equation*}
Note $\beta(p,\omega):=\min\{-\omega+2p+1,1/p+1\}$. In particular, by (\ref{axis-int}), (\ref{u1_endpoint}) and (\ref{anyN}), we have
  \begin{equation*}\label{axis-p=2}\vspace{-2pt}
    J_{\omega}^{[2]}(x_{\mathrm{axis}})=\sqrt{\frac{2}{\pi r}}\cos\Bigl(r-\frac{2\omega+1}{4}\pi\Bigr)+\mathcal{O}(r^{-\frac{3}{2}})\qquad\text{as }r\to\infty.
  \end{equation*}
  This recovers the classical asymptotic form (\ref{J-asym.}), illustrating that the method is consistent with known results, since
$J_{\omega}^{[2]}(x_{\mathrm{axis}})=\mathcal{J}_{\omega,\pi/2}^{[2]}(r)=J_{\omega}(r)$.
\end{proof}

\begin{remark}
The asymptotic formula above reveals a fundamental difference between the $p$-Bessel functions with $p<1$ and the classical Bessel functions. In contrast to the classical case, the asymptotic behavior on the coordinate axes depends explicitly on the order $\omega$. Indeed, the contribution from the endpoint $u=0$ has order $r^{\omega-p-1}$, whereas the oscillatory contribution from the endpoint $u=1$ has amplitude of order $r^{-1/p}$. Thus, the contribution from $u=0$ is of larger order than the oscillatory contribution from $u=1$ when $\omega>p+1-1/p$.\par
For $2/p=3$, that is, $p=2/3$, this condition becomes $\omega>1/6$. On the other hand, if $2/p\geq4$, then $p\leq1/2$ and hence $p+1-1/p\leq1/2+1-2=-1/2<0$. Therefore, for $2/p\geq4$, the contribution from $u=0$ is of larger order than the oscillatory contribution from $u=1$ for every $\omega\geq0$. This order-dependent behavior is in sharp contrast to the order-independent leading decay of the classical Bessel functions.\par
This behavior also reflects the geometry of the associated $p$-circle. In particular, the cusps on the coordinate axes are reflected analytically in the non-smoothness of the amplitude function at the corresponding endpoint. The resulting fractional-power term in the local expansion of the amplitude is precisely what produces the additional order-dependent contribution.\par
The present analysis is restricted to the coordinate axes. To obtain a uniform asymptotic description with respect to the distorted angle $\varphi$, a more delicate analysis near the axes will be necessary. One possible approach is to introduce a suitable scaling of the spatial parameter, such as $x=(\delta\lambda,\lambda)$ with $\delta$ small and $\lambda\to\infty$, and to analyze the resulting oscillatory integral (\ref{J-int0>0}) and (\ref{J-int>0}) in the transition region between the axis and the non-axis directions. Another possible approach is to use the integral representations given in (2.6) and (2.7) of \cite{K3}
  \begin{equation*}
    J_{\omega}^{[p]}(x)=\frac{|x|_{p}^{\omega}}{p^{\omega-1}\Gamma(\omega)}\int_{0}^{1}J_{0}^{[p]}(\tau x)\tau(1-\tau^{p})^{\omega-1}d\tau\quad\text{for }x\in\mathbb{R}^{2}\setminus{\{0\}}, \omega>0
  \end{equation*}\vspace{-5pt}
  \begin{equation*}
    \mathcal{J}_{\omega,\varphi}^{[p]}(r)=\frac{r^{\omega}}{p^{\omega-1}\Gamma(\omega)}\int_{0}^{1}\mathcal{J}_{0,\varphi}^{[p]}(\tau r)\tau(1-\tau^{p})^{\omega-1}d\tau\quad\text{for }r>0, \omega>0
  \end{equation*}
and to combine this with uniform asymptotic information for $J_{0}^{[p]}$ (Proposition \ref{uni-zero-ord}). These approaches may provide a basis for a more complete angularly uniform asymptotic analysis.
\end{remark}
  
\section{Extension to the Complex Domain}
  \hspace{13pt}In this section, we extend the $p$-Bessel function $\mathcal{J}_{\omega,\varphi}^{[p]}$, originally defined on the nonnegative real axis for parameters $p$ satisfying $2/p\in\mathbb{N}$, to a function on $\mathbb{C}\setminus(-\infty,0]$ with complex order $\omega$ such that $\mathrm{Re}(\omega)>0$ or $\omega=0$. To this end, motivated by the form of the series representation (\ref{2-J}) and (\ref{p-Bessel-S})
  \begin{align*}
      \mathcal{J}_{\omega,\varphi}^{[p]}(r)&:=\frac{(\frac{2}{p})^{2}}{p^{\omega}\Gamma(\frac{1}{p})^{2}}\sum_{k=0}^{\infty}\frac{(-1)^{k}\ r^{2k+\omega}}{\Gamma(\frac{2}{p}(k+1)+\omega)}\sum_{m_{1}+m_{2}=k}\frac{\Gamma(\frac{2m_{1}+1}{p})\Gamma(\frac{2m_{2}+1}{p})}{(2m_{1})!\ (2m_{2})!}(\cos^{2m_{1}}\varphi\sin^{2m_{2}}\varphi)^{\frac{2}{p}}\\
      &=\Bigl(\frac{2}{p}\Bigr)^{2+\omega}\frac{\pi}{\Gamma(\frac{1}{p})^{2}}\sum_{k=0}^{\infty}\frac{(-1)^{k}}{k!\Gamma(\frac{2}{p}(k+1)+\omega)}\Bigl(\frac{r}{2}\Bigr)^{2k+\omega}\Phi_{k,\varphi}^{[p]}\quad\text{for }r\geq0,
  \end{align*}
  \begin{equation*}
      \text{with }\Phi_{k,\varphi}^{[p]} := \sum_{n=0}^{k} \frac{\bigl(\frac{2}{p}(k+1)-1\bigr)!}{n!\,(k-n)!} \left( \frac{B(\frac{2}{p}(n+\frac{1}{2}), \frac{2}{p}(k-n+\frac{1}{2}))}{B(n+\frac{1}{2}, k-n+\frac{1}{2})} \right) (\cos^{\frac{4}{p}}\varphi)^n (\sin^{\frac{4}{p}}\varphi)^{k-n},
  \end{equation*}
  we redefine the complex $p$-Bessel function of order $\omega$ (with $\mathrm{Re}(\omega)>0$ or $\omega=0$) by the series 
  \begin{align}\label{mathcalJz-1}
      \mathcal{J}_{\omega,\varphi}^{[p]}(z)&:=\frac{(\frac{2}{p})^{2}}{p^{\omega}\Gamma(\frac{1}{p})^{2}}\sum_{k=0}^{\infty}\frac{(-1)^{k}\ z^{2k+\omega}}{\Gamma(\frac{2}{p}(k+1)+\omega)}\sum_{m_{1}+m_{2}=k}\frac{\Gamma(\frac{2m_{1}+1}{p})\Gamma(\frac{2m_{2}+1}{p})}{(2m_{1})!\ (2m_{2})!}(\cos^{2m_{1}}\varphi\sin^{2m_{2}}\varphi)^{\frac{2}{p}}\\
      &=\Bigl(\frac{2}{p}\Bigr)^{2+\omega}\frac{\pi}{\Gamma(\frac{1}{p})^{2}}\sum_{k=0}^{\infty}\frac{(-1)^{k}}{k!\Gamma(\frac{2}{p}(k+1)+\omega)}\Bigl(\frac{z}{2}\Bigr)^{2k+\omega}\Phi_{k,\varphi}^{[p]}. \label{mathcalJz-2}
  \end{align}
  Note that from the expression (\ref{mathcalJz-2}), it is clear that $\mathcal{J}_{\omega,\varphi}^{[2]}(z)=J_{\omega}(z)$. \par
  Let $\mathbb{C}\setminus(-\infty,0]$ denote the complex plane cut along the negative real axis. On this domain we define the principal branch of the logarithm by
$\mathrm{Log} z = \log|z| + i\mathrm{Arg}(z)$, $\mathrm{Arg}(z)\in(-\pi,\pi]$, and we define $z^{\alpha} := e^{\alpha \mathrm{Log} z}$ for $\alpha\in\mathbb{C}$.\par
  Let $\mathbb{R}_{-}:=(-\infty,0]$. Then $\mathcal{J}_{\omega,\varphi}^{[p]}(z)$ is holomorphic on $\mathbb{C}\setminus{\mathbb{R}_{-}}$. \par\vspace{15pt}

  Now, we verify this holomorphy. Recall that $\mathrm{Log} z$ and $z^{\alpha}$ are holomorphic on $\mathbb{C}\setminus{\mathbb{R}_{-}}$; it follows that the sequence of partial sums of the series defining $\mathcal{J}_{\omega,\varphi}^{[p]}(z)$ is also holomorphic on the same domain. \par
  Next, for any compact set $V\subset\mathbb{C}\setminus{\mathbb{R}_{-}}$, let $M_{V}:=\max_{z\in V}|z|$. Writing $\omega=\omega_{1}+i\omega_{2}(\in\mathbb{C})$ and noting the following, we then have $|z^{\omega}|\leq M_{V}^{\omega_{1}}e^{|\omega_{2}|\pi}$ by assumption, $\omega_{1}\geq0$.
  \begin{equation*}
    |z^{\omega}|=|e^{\omega\mathrm{Log}z}|=|e^{\omega\log |z|+i\omega\mathrm{Arg}(z)}|=|e^{\omega_{1}\log |z|}\cdot e^{-\omega_{2}\mathrm{Arg}(z)}|=|z|^{\omega_{1}}e^{-\omega_{2}\mathrm{Arg}(z)}.
  \end{equation*}\par
  It remains to show that the series representation (\ref{mathcalJz-1}) of $\mathcal{J}_{\omega,\varphi}^{[p]}(z)$ converges uniformly on $V$, more precisely, that the rearranged series (with the order of summation interchanged) converges absolutely on $V$. Once this is established, the holomorphy of $\mathcal{J}_{\omega,\varphi}^{[p]}(z)$ follows from the preservation of holomorphy under uniformly convergent sequences of holomorphic functions. \par
  Noting the validity of inequalities
  \begin{equation*}
    \frac{\Gamma(n+\frac{k}{2})\Gamma(m+\frac{k}{2})}{\Gamma(n+m+k)}\leq\frac{\Gamma(\frac{k}{2})^{2}}{\Gamma(k)},\qquad \Gamma(k)|\omega\Gamma(\omega)|\leq|\Gamma(k+\omega)|\ (=(\textstyle\prod_{l=1}^{k-1}|l+\omega|)\ |\Gamma(\omega+1)|)
  \end{equation*}
  \begin{equation*}
  (\text{from } |l+\omega|^{2}-l^{2}=(l+\omega)(l+\overline{\omega})-l^{2}=l(\omega+\overline{\omega})+|\omega|^{2}=2l\omega_{1}+|\omega|^{2}\geq0)
  \end{equation*}
  for $k\in\mathbb{N},\ n,m\in\mathbb{N}_{0}$, and $\omega\in\mathbb{C}$ with $\mathrm{Re}(\omega)>0$, and defining $C(\omega):=1/|\omega\Gamma(\omega)|$ ($\mathrm{Re}(\omega)>0$), 1 ($\omega=0$), we see that, for $(m_{1},m_{2})\in\mathbb{N}_{0}^{2}$, the expression in 
  \begin{equation*}
    \biggl|\frac{\Gamma(\frac{2}{p}m_{1}+\frac{1}{p})\Gamma(\frac{2}{p}m_{2}+\frac{1}{p})}{\Gamma(\frac{2}{p}(m_{1}+m_{2}+1)+\omega)}\biggr|\leq C(\omega)\frac{\Gamma(\frac{2}{p}m_{1}+\frac{1}{p})\Gamma(\frac{2}{p}m_{2}+\frac{1}{p})}{\Gamma(\frac{2}{p}(m_{1}+m_{2}+1))}\leq C(\omega)\frac{\Gamma(\frac{1}{p})^{2}}{\Gamma(\frac{2}{p})}
  \end{equation*}
  can be bounded from above by a quantity depending only on $\omega$ and $p$, independent of $(m_{1},m_{2})$. Hence, the following result holds.
  \begin{align*}
    \sum_{m_{1}=0}^{\infty}\sum_{m_{2}=0}^{\infty}&\ \biggl|\frac{(\frac{2}{p})^{2}(-1)^{m_{1}+m_{2}}z^{2(m_{1}+m_{2})+\omega}}{p^{\omega}\Gamma(\frac{1}{p})^{2}\Gamma(\frac{2}{p}(m_{1}+m_{2}+1)+\omega)}\frac{\Gamma(\frac{2m_{1}+1}{p})\Gamma(\frac{2m_{2}+1}{p})}{(2m_{1})!\ (2m_{2})!}|\cos^{2m_{1}}\varphi\sin^{2m_{2}}\varphi|^{\frac{2}{p}}\biggr|\\
      &=\frac{4|z^{\omega}|}{|p^{\omega+2}|\Gamma(\frac{1}{p})^{2}}\sum_{m_{1}=0}^{\infty}\sum_{m_{2}=0}^{\infty}\ \biggl|\frac{\Gamma(\frac{2}{p}m_{1}+\frac{1}{p})\Gamma(\frac{2}{p}m_{2}+\frac{1}{p})}{\Gamma(\frac{2}{p}(m_{1}+m_{2}+1)+\omega)}\biggr|\frac{(|z||\cos\varphi|^{\frac{2}{p}})^{2m_{1}}(|z||\sin\varphi|^{\frac{2}{p}})^{2m_{2}}}{(2m_{1})!\ (2m_{2})!}\\
      &\leq\frac{4M_{V}^{\omega_{1}}e^{|\omega_{2}|\pi}}{p^{\omega_{1}+2}\Gamma(\frac{1}{p})^{2}}\cdot C(\omega)\frac{\Gamma(\frac{1}{p})^{2}}{\Gamma(\frac{2}{p})}\biggl(\sum_{m_{1}=0}^{\infty}\frac{(|z||\cos\varphi|^{\frac{2}{p}})^{2m_{1}}}{(2m_{1})!}\biggr)\biggl(\sum_{m_{2}=0}^{\infty}\frac{(|z||\sin\varphi|^{\frac{2}{p}})^{2m_{2}}}{(2m_{2})!}\biggr)\\
      &\leq\frac{4M_{V}^{\omega_{1}}e^{|\omega_{2}|\pi}C(\omega)}{p^{\omega_{1}+2}\Gamma(\frac{2}{p})}\ e^{M_{V}(|\cos\varphi|^{\frac{2}{p}}+|\sin\varphi|^{\frac{2}{p}})}<+\infty.
  \end{align*}
  \par
  Therefore, since the series defining $\mathcal{J}_{\omega,\varphi}^{[p]}$ converges uniformly on $V$, the desired holomorphy of $\mathcal{J}_{\omega,\varphi}^{[p]}$ on $\mathbb{C}\setminus\mathbb{R}_{-}$ is established. 
  \begin{remark}
  The results presented in Theorems \ref{E-K-derJ} and \ref{matcal-J-int0} are stated for real arguments $r\geq0$. We note, however, that both the Erd\'{e}lyi–Kober-type fractional derivative relations and the integral representation of $\mathcal{J}_{\omega,\varphi}^{[p]}$ naturally extend to the complex domain $z\in\mathbb{C}\setminus\mathbb{R}_{-}$. Indeed, the defining series and the integral representation involve only holomorphic functions and linear operations (sums, integration, differentiation), so that all identities derived in the real setting continue to hold in this complex domain.
  \end{remark}
  
  \subsection{\normalsize{Definitions of the p-Cosine and p-Sine Functions}}
  \hspace{13pt}In this subsection, we focus on the identity $J_{-1/2}(z)=\sqrt{2/\pi}z^{-1/2}\cos z$. In particular, we consider values of $p$ for which $2/p$ is a positive odd integer (equivalently, $(1/p)-(1/2)\in\mathbb{N}_{0}$). Under this assumption, we define the $p$-cosine functions (and $p$-sine functions) corresponding to the $p$-Bessel functions. Note that in the previous subsections, the order $\omega$ was assumed to be a non-negative real number. \par
  First, observing that $\sqrt{\pi}(2l-1)!!=2^{l}\Gamma(l+\frac{1}{2})$, we define the $p$-cosine function from the series
  \begin{align*}
    \sum_{k=0}^{\infty}\frac{(-1)^{k}}{k!\ \Gamma\bigl(\frac{2}{p}(k+1)-\frac{1}{p}\bigr)}\Bigl(\frac{z}{2}\Bigr)^{2k-\frac{1}{p}}\Phi_{k,\varphi}^{[p]}&=\Bigl(\frac{2}{z}\Bigr)^{\frac{1}{p}}\sum_{k=0}^{\infty}\frac{(-1)^{k}}{2^{k}k!}\biggl(\frac{\Phi_{k,\varphi}^{[p]}}{\Gamma\bigl(\frac{1}{p}(2k+1)\bigr)2^{k}}\biggr)z^{2k}\\
    &=\frac{2^{\frac{1}{p}}}{\sqrt{\pi}z^{\frac{1}{p}}}\sum_{k=0}^{\infty}\frac{(-1)^{k}}{(2k)!!}\biggl(\frac{\sqrt{\pi}\ \Phi_{k,\varphi}^{[p]}}{\Gamma\bigl(\frac{1}{p}(2k+1)\bigr)2^{k}}\biggr)z^{2k}
  \end{align*}
  as follows.
  \begin{align}\label{cos-p-1}
  \cos_{\varphi}^{[p]}(z)&:=\sum_{k=0}^{\infty}\frac{(-1)^{k}}{(2k)!!}\biggl(\frac{\sqrt{\pi}\ \Phi_{k,\varphi}^{[p]}}{\Gamma\bigl(\frac{1}{p}(2k+1)\bigr)2^{k}}\biggr)z^{2k}\\
  &=\sum_{k=0}^{\infty}\frac{(-1)^{k}\ \Phi_{k,\varphi}^{[p]}}{(2k)!!\ 2^{k}}\biggl(\frac{2^{\frac{1}{p}(2k+1)-\frac{1}{2}}}{\bigl(2\bigl(\frac{1}{p}(2k+1)-\frac{1}{2}\bigr)-1\bigr)!!}\biggr)z^{2k}\notag\\
  &=\sum_{k=0}^{\infty}\frac{(-1)^{k}}{\bigl(\frac{2}{p}(2k+1)-1\bigr)!}\biggl(\frac{2^{\frac{2}{p}k+\frac{1}{p}-\frac{1}{2}}\bigl(\frac{2}{p}(2k+1)-1\bigr)!!\ \Phi_{k,\varphi}^{[p]}}{2^{k}\ (2k)!!}\biggr)z^{2k}\notag\\
  &=\sum_{k=0}^{\infty}\frac{(-1)^{k}}{\bigl(\frac{2}{p}(2k+1)-1\bigr)!}\biggl(\frac{2^{(\frac{2}{p}-1)(2k+1)}\bigl(\frac{2}{p}(k+\frac{1}{2})-\frac{1}{2}\bigr)!\ \Phi_{k,\varphi}^{[p]}}{k!}\biggr)z^{2k}.\label{cos-p-2}
  \end{align}
  Next, we verify that $\cos_{\varphi}^{[p]}(z)$ is holomorphic on $\mathbb{C}$. \par
  Indeed, since the distorted angle coefficient (\ref{Phi-co}) can be expressed as
  \begin{equation}\label{Phi-0}
    \Phi_{k,\varphi}^{[p]}=\frac{k!\ 2^{2k}}{\pi}\sum_{m_{1}+m_{2}=k}^{m\in\mathbb{N}_{0}^{2}}\frac{\Gamma(\frac{2}{p}m_{1}+\frac{1}{p})\Gamma(\frac{2}{p}m_{2}+\frac{1}{p})}{(2m_{1})!\ (2m_{2})!}(\cos^{m_{1}}\varphi\sin^{m_{2}}\varphi)^{\frac{4}{p}},
  \end{equation}
  and the two inequalities
  \begin{equation*}
  \frac{(\frac{2}{p}(m_{1}+m_{2})+\frac{1}{p}-\frac{1}{2})!}{(2(\frac{2}{p}(m_{1}+m_{2})+\frac{1}{p}-\frac{1}{2}))!}\leq\frac{1}{(\frac{2}{p}(m_{1}+m_{2})+\frac{1}{p}-\frac{1}{2})!}=\frac{1}{\Gamma(\frac{2}{p}(m_{1}+m_{2})+\frac{1}{p}+\frac{1}{2})},
  \end{equation*}
  \begin{equation}\label{leq1}
  \frac{\Gamma(\frac{2}{p}m_{1}+\frac{1}{p})\Gamma(\frac{2}{p}m_{2}+\frac{1}{p})}{\Gamma(\frac{2}{p}(m_{1}+m_{2})+\frac{1}{p}+\frac{1}{2})}\leq\frac{\Gamma(\frac{1}{p})^{2}}{\Gamma(\frac{1}{p}+\frac{1}{2})}
  \end{equation}
  hold, we can show that the following series, obtained by rearranging the order of summation in the series (\ref{cos-p-2}), converges uniformly on compact subsets of $\mathbb{C}$. Hence the series converges locally uniformly on $\mathbb{C}$, and therefore defines an entire function. 
  \begin{align*}
    \sum^{\infty}_{m_{1}=0}\sum^{\infty}_{m_{2}=0}&\ \biggl|\frac{(-1)^{m_{1}+m_{2}}}{\bigl(2\cdot\frac{2}{p}(m_{1}+m_{2})+\frac{2}{p}-1\bigr)!}\biggl(\frac{2^{(\frac{2}{p}-1)(2(m_{1}+m_{2})+1)}\bigl(\frac{2}{p}(m_{1}+m_{2})+\frac{1}{p}-\frac{1}{2}\bigr)!}{(m_{1}+m_{2})!}\biggr)\\
    &\times\Bigl(\frac{(m_{1}+m_{2})!\ 2^{2(m_{1}+m_{2})}}{\pi}\Bigr)\frac{\Gamma(\frac{2}{p}m_{1}+\frac{1}{p})\Gamma(\frac{2}{p}m_{2}+\frac{1}{p})}{(2m_{1})!\ (2m_{2})!}(\cos^{m_{1}}\varphi\sin^{m_{2}}\varphi)^{\frac{4}{p}}z^{2(m_{1}+m_{2})}\biggr|\\
    &\leq\frac{2^{\frac{2}{p}-1}\Gamma(\frac{1}{p})^{2}}{\pi\Gamma(\frac{1}{p}+\frac{1}{2})}\biggl(\sum_{m_{1}=0}^{\infty}\frac{|2^{\frac{2}{p}}z|^{2m_{1}}}{(2m_{1})!}\biggr)\biggl(\sum_{m_{2}=0}^{\infty}\frac{|2^{\frac{2}{p}}z|^{2m_{2}}}{(2m_{2})!}\biggr)
    \leq\frac{2^{\frac{2}{p}-1}\Gamma(\frac{1}{p})^{2}}{\pi\Gamma(\frac{1}{p}+\frac{1}{2})}e^{2^{\frac{2}{p}+1}|z|}<+\infty.
  \end{align*}\par
  Note that (\ref{leq1}) is a special case of the following inequality for $k\in\mathbb{N},\ n,m\in\mathbb{N}_{0}$.
  \begin{equation*}\vspace{-5pt}
    \frac{\Gamma(n+\frac{k}{2})\Gamma(m+\frac{k}{2})}{\Gamma(n+m+\frac{k}{2}+\frac{1}{2})}\leq\frac{\Gamma(\frac{k}{2})^{2}}{\Gamma(\frac{k}{2}+\frac{1}{2})}.
  \end{equation*}
  \par\vspace{15pt}
Using this entire function together with the series representation (\ref{mathcalJz-2}), we obtain the $p$-Bessel function of order $-1/p$ that is holomorphic on $\mathbb{C}\setminus\mathbb{R}_{-}$.
  \begin{equation}\label{(-1/p)mathcalJ}
    \mathcal{J}_{-\frac{1}{p},\varphi}^{[p]}(z):=\frac{\pi\bigl(\frac{2}{p}\bigr)^{2-\frac{1}{p}}}{\Gamma(\frac{1}{p})^{2}}\cdot\frac{2^{\frac{1}{p}}}{\sqrt{\pi}z^{\frac{1}{p}}}\cos_{\varphi}^{[p]}(z)=\frac{4\sqrt{\pi}}{p^{2-\frac{1}{p}}\Gamma\bigl(\frac{1}{p}\bigr)^{2}}\frac{\cos_{\varphi}^{[p]}(z)}{z^{\frac{1}{p}}}.
  \end{equation}
  Indeed, it satisfies $\mathcal{J}_{-1/2,\varphi}^{[2]}(z)=\sqrt{2/\pi}z^{-1/2}\cos z=J_{-1/2}(z)$. \par
  On the other hand, in view of this property together with the identity $J_{1/2}(z)=\sqrt{2/\pi}z^{-\frac{1}{2}}\sin z$, we define the $p$-sine function as follows.
  \begin{equation}\label{sin-p-1}
    \sin_{\varphi}^{[p]}(z):=\sum_{k=0}^{\infty}\frac{(-1)^{k}}{(2k)!!}\biggl(\frac{\sqrt{\pi}\ \Phi_{k,\varphi}^{[p]}}{\Gamma\bigl(\frac{1}{p}(2k+3)\bigr)2^{k+1}}\biggr)z^{2k+1}.
  \end{equation}
  \begin{equation*}
    \biggl(\text{In particular, }\sin_{\varphi}^{[2]}(z)=\sum_{k=0}^{\infty}\frac{(-1)^{k}}{(2k)!!}\biggl(\frac{\sqrt{\pi}\cdot 1}{\Gamma\bigl(k+\frac{3}{2}\bigr)2^{k+1}}\biggr)z^{2k+1}=\sum_{k=0}^{\infty}\frac{(-1)^{k}}{(2k)!!\ (2k+1)!!}z^{2k+1}=\sin z.\ \biggr)
  \end{equation*}\par
  Indeed, the validity of this definition becomes clear by considering the representation of $\mathcal{J}_{1/p,\varphi}^{[p]}$.
  \begin{align*}
    \sum_{k=0}^{\infty}\frac{(-1)^{k}}{k!\ \Gamma\bigl(\frac{2}{p}(k+1)+\frac{1}{p}\bigr)}\Bigl(\frac{z}{2}\Bigr)^{2k+\frac{1}{p}}\Phi_{k,\varphi}^{[p]}&=\frac{1}{z}\Bigl(\frac{z}{2}\Bigr)^{\frac{1}{p}}\sum_{k=0}^{\infty}\frac{(-1)^{k}}{2^{k}k!}\biggl(\frac{\Phi_{k,\varphi}^{[p]}}{\Gamma\bigl(\frac{1}{p}(2k+3)\bigr)2^{k}}\biggr)z^{2k+1}\\
    &=\frac{z^{\frac{1}{p}-1}}{\sqrt{\pi}2^{\frac{1}{p}-1}}\sum_{k=0}^{\infty}\frac{(-1)^{k}}{(2k)!!}\biggl(\frac{\sqrt{\pi}\ \Phi_{k,\varphi}^{[p]}}{\Gamma\bigl(\frac{1}{p}(2k+3)\bigr)2^{k+1}}\biggr)z^{2k+1}\\
    &=\frac{z^{\frac{1}{p}-1}}{\sqrt{\pi}2^{\frac{1}{p}-1}}\sin_{\varphi}^{[p]}(z),
  \end{align*}
  \begin{equation*}\label{(1/p)mathcalJ}
    \mathcal{J}_{\frac{1}{p},\varphi}^{[p]}(z)=\frac{\pi\bigl(\frac{2}{p}\bigr)^{2+\frac{1}{p}}}{\Gamma(\frac{1}{p})^{2}}\cdot\frac{z^{\frac{1}{p}-1}}{\sqrt{\pi}2^{\frac{1}{p}-1}}\sin_{\varphi}^{[p]}(z)=\frac{8\sqrt{\pi}}{p^{2+\frac{1}{p}}\Gamma\bigl(\frac{1}{p}\bigr)^{2}}z^{\frac{1}{p}-1}\sin_{\varphi}^{[p]}(z).
  \end{equation*}
  Moreover, from the form $z^{1-(1/p)}\mathcal{J}_{1/p,\varphi}^{[p]}(z)$, it also follows that $\sin_{\varphi}^{[p]}$ is an entire function.
  
  \subsection{\normalsize{Poisson-type Integral Representation (Proof of Theorem \ref{GPIR})}}
  \hspace{13pt}In this subsection, we derive a counterpart for the $p$-Bessel functions $\mathcal{J}_{\omega,\varphi}^{[p]}$ corresponding to the Poisson integral representation for the Bessel function $J_{\omega}$ (see \cite[3·3, (1)]{Watson})
  \begin{equation*}\label{PI}
    J_{\omega}(z)=\frac{1}{\sqrt{\pi}\Gamma(\omega+\frac{1}{2})}\Bigl(\frac{z}{2}\Bigr)^{\omega}\int_{0}^{\pi}\cos(z\cos\theta)\sin^{2\omega}\theta d\theta\quad\text{for }\omega>-\frac{1}{2}.
  \end{equation*}
  \begin{proof}[Proof of Theorem \ref{GPIR}]
  We first transform the series representation (\ref{mathcalJz-2}) as follows. 
  \begin{align*}
  \mathcal{J}_{\omega,\varphi}^{[p]}(z)&=\frac{\pi(\frac{2}{p})^{2+\omega}}{\Gamma(\frac{1}{p})^{2}}\sum_{k=0}^{\infty}\frac{(-1)^{k}}{k!\Gamma(\frac{2}{p}(k+1)+\omega)}\Bigl(\frac{z}{2}\Bigr)^{2k+\omega}\Phi_{k,\varphi}^{[p]}\\
  &=\frac{\pi(\frac{2}{p})^{2+\omega}}{\Gamma(\frac{1}{p})^{2}}\Bigl(\frac{z}{2}\Bigr)^{\omega}\sum_{k=0}^{\infty}\frac{(-1)^{k}\Phi_{k,\varphi}^{[p]}}{2^{2k}k!}z^{2k}\biggl(\frac{B(\frac{2}{p}k+\frac{1}{p},\omega+\frac{1}{p})}{\Gamma(\omega+\frac{1}{p})\Gamma(\frac{2}{p}k+\frac{1}{p})}\biggr)\\
  &=\frac{\pi(\frac{2}{p})^{2+\omega}}{\Gamma(\frac{1}{p})^{2}\Gamma(\omega+\frac{1}{p})}\Bigl(\frac{z}{2}\Bigr)^{\omega}\sum_{k=0}^{\infty}\frac{(-1)^{k}\Phi_{k,\varphi}^{[p]}}{k!}z^{2k}\biggl(\frac{2^{-2k}}{\Gamma(\frac{2}{p}k+\frac{1}{p})}\biggr)\int_{0}^{1}t^{\frac{2}{p}k+\frac{1}{p}-1}(1-t)^{\omega+\frac{1}{p}-1}dt\\
  &=\frac{\pi(\frac{2}{p})^{2+\omega}}{\Gamma(\frac{1}{p})^{2}\Gamma(\omega+\frac{1}{p})}\Bigl(\frac{z}{2}\Bigr)^{\omega}\sum_{k=0}^{\infty}\frac{(-1)^{k}}{(2k)!!}\biggl(\frac{\Phi_{k,\varphi}^{[p]}}{\Gamma(\frac{1}{p}(2k+1))2^{k}}\biggr)z^{2k}p\int_{0}^{1}s^{2k}(1-s^{p})^{\omega+\frac{1}{p}-1}ds\\
  &=\frac{\sqrt{\pi}(\frac{2}{p})^{2+\omega}p}{\Gamma(\frac{1}{p})^{2}\Gamma(\omega+\frac{1}{p})}\Bigl(\frac{z}{2}\Bigr)^{\omega}\int_{0}^{1}\biggl(
  \sum_{k=0}^{\infty}\frac{(-1)^{k}}{(2k)!!}\Bigl(\frac{\sqrt{\pi}\Phi_{k,\varphi}^{[p]}}{\Gamma(\frac{1}{p}(2k+1))2^{k}}\Bigr)(zs)^{2k}\biggr)(1-s^{p})^{\omega+\frac{1}{p}-1}ds.
  \end{align*}
  From this, we obtain the following integral representation involving the entire function $\cos_{\varphi}^{[p]}$ defined above (note that in the final equality, term-by-term integration is justified, since the series converges uniformly with respect to $s \in [0,1]$ for $z$ in compact subsets of $\mathbb{C}$). \vspace{-5pt}
  \begin{equation*}
  \mathcal{J}_{\omega,\varphi}^{[p]}(z)=\frac{\sqrt{\pi}(\frac{2}{p})^{2+\omega}p}{\Gamma(\frac{1}{p})^{2}\Gamma(\omega+\frac{1}{p})}\Bigl(\frac{z}{2}\Bigr)^{\omega}\int_{0}^{1}\cos_{\varphi}^{[p]}(zs)(1-s^{p})^{\omega+\frac{1}{p}-1}ds.
  \end{equation*}\par
  Next, by making the change of variables $s=\cos^{\frac{2}{p}}\theta$, the expression can be written as
  \begin{align*}
    \int_{0}^{1}\cos_{\varphi}^{[p]}(zs)(1-s^{p})^{\omega+\frac{1}{p}-1}ds&=\int_{\frac{\pi}{2}}^{0}\cos_{\varphi}^{[p]}(z\cos^{\frac{2}{p}}\theta)\cdot(1-\cos^{2}\theta)^{\omega+\frac{1}{p}-1}\biggl(-\frac{2}{p}\sin\theta\cos^{\frac{2}{p}-1}\theta d\theta\biggr)\\
    &=\frac{2}{p}\int_{0}^{\frac{\pi}{2}}\cos_{\varphi}^{[p]}(z\cos^{\frac{2}{p}}\theta)\cdot(\sin\theta)^{2\omega+\frac{2}{p}-2}(\sin\theta\cos^{\frac{2}{p}-1}\theta)d\theta\\
    &=\frac{2}{p}\int_{0}^{\frac{\pi}{2}}\cos_{\varphi}^{[p]}(z\cos^{\frac{2}{p}}\theta)\sin^{2\omega}\theta(\cos\theta\sin\theta)^{\frac{2}{p}-1}d\theta.
  \end{align*}
  Observing that the resulting integrand does not necessarily become an even function with respect to $\theta$ depending on $\omega$ (for example, when $\omega=1/4$, the factor $\sin^{2\omega}\theta$ becomes $\sqrt{\sin\theta}$, so the interval of integration cannot be extended to $[-\frac{\pi}{2},0]$), we obtain the desired integral representations as follows. In particular, the latter representation corresponds to the case $\omega=n\in\mathbb{N}_{0}$. 
  \begin{equation*}
    \mathcal{J}_{\omega,\varphi}^{[p]}(z)=\frac{\sqrt{\pi}(\frac{2}{p})^{2+\omega}2}{\Gamma(\frac{1}{p})^{2}\Gamma(\omega+\frac{1}{p})}\Bigl(\frac{z}{2}\Bigr)^{\omega}\int_{0}^{\frac{\pi}{2}}\cos_{\varphi}^{[p]}(z\cos^{\frac{2}{p}}\theta)\sin^{2\omega}\theta(\cos\theta\sin\theta)^{\frac{2}{p}-1}d\theta,
  \end{equation*}
  \begin{equation*}
    \mathcal{J}_{n,\varphi}^{[p]}(z)=\frac{\sqrt{\pi}(\frac{2}{p})^{2+n}}{\Gamma(\frac{1}{p})^{2}\Gamma(n+\frac{1}{p})}\Bigl(\frac{z}{2}\Bigr)^{n}\int_{-\frac{\pi}{2}}^{\frac{\pi}{2}}\cos_{\varphi}^{[p]}(z\cos^{\frac{2}{p}}\theta)\sin^{2n}\theta(\cos\theta\sin\theta)^{\frac{2}{p}-1}d\theta.
  \end{equation*}
\end{proof}
\begin{remark}
Although the $p$-Bessel function corresponds to a non-circular domain with cusps—namely, the astroid-type $p$-circle—it is evident from previous work and the results in the preceding section that it retains properties closely analogous to those of the classical Bessel function. On the other hand, a noteworthy aspect of the present section is that, through complex extension, a new entire function, the $p$-cosine function, has been introduced. This development allows the $p$-Bessel functions to be represented as a single-variable integral rather than a multiple integral. This advancement suggests that, by following the precedent set in the Gauss circle problem—where properties of the classical Bessel function were exploited to tackle the lattice point problem—complex-analytic methods may provide a promising framework for addressing the corresponding lattice point problem in the astroid-type $p$-circle in future work.
\end{remark}

\vspace{20pt}\hspace{-15pt}\textbf{Data availability}\quad
  No data was used for the research described in the article. \vspace{10pt}\\
\textbf{Acknowledgements}\quad
  The author is deeply grateful to Prof. Mitsuru Sugimoto for his invaluable guidance, numerous constructive suggestions, and insightful remarks on harmonic analysis. This work was also financially supported by JST SPRING, Grant Number JPMJSP2125, and the author would like to take this opportunity to thank the ``THERS Make New Standards Program for the Next Generation Researchers'' for providing excellent research conditions during the preparation of this paper.\vspace{10pt}\\
\textbf{Disclosure statement}\quad
  No potential conflict of interest was reported by the author.

  \itshape{
  \hspace{13pt}The author's affiliation: Graduate School of Mathematics, Nagoya University, Chikusa-ku, Nagoya 464-8602, Japan\par
  The author's email address: kitajima.masaya.p2@a.mail.nagoya-u.ac.jp}
\end{document}